\documentclass[11pt,letterpaper]{article}
\usepackage[letterpaper, portrait, margin=1in]{geometry}
\usepackage[colorlinks, citecolor = blue, linkcolor = magenta]{hyperref}
\usepackage{algorithm}
\usepackage[noend]{algpseudocode}
\usepackage{url}
\usepackage{amsmath,amssymb,amsthm}
\usepackage{thmtools,thm-restate}
\usepackage[noabbrev,capitalise,nameinlink]{cleveref}
\usepackage{mathtools}
\usepackage{xspace}
\usepackage{verbatim}
\usepackage{mathrsfs}
\usepackage[dvipsnames,svgnames,table]{xcolor}
\usepackage{pgf}
\usepackage[dvipsnames]{xcolor}
\usepackage{todo}
\usepackage{tabularx}
\usepackage{enumitem}
\usepackage{derivative}
\usepackage{bm}
\usepackage{multirow}
\usepackage{diagbox}
\usepackage{nicematrix}
\usepackage{tikz}
\usetikzlibrary{shapes.geometric, arrows, arrows.meta, positioning}

\tikzstyle{node} = [rectangle, rounded corners, minimum width=2cm, minimum height=1cm, text centered, draw=black, fill=blue!10]
\tikzstyle{hidden} = [draw=none, fill=none] 
\tikzstyle{double arrow} = [thick, double distance=2pt, -{Latex[round]}]

\usepackage{dsfont}

\newcommand*\ie{i.\kern.1em e.\ }
\newcommand*\eg{e.\kern.1em g.\ }
\newcommand*\cf{c.\kern.1em f.\ }

\theoremstyle{plain}
\newtheorem{theorem}{Theorem}[section]
\newtheorem{lemma}[theorem]{Lemma}

\newtheorem{proposition}[theorem]{Proposition}
\newtheorem{claim}[theorem]{Claim}
\newtheorem{corollary}[theorem]{Corollary}
\newtheorem{question}[theorem]{Question}

\newtheorem{observation}[theorem]{Observation}

\crefname{claim}{Claim}{Claims}

\theoremstyle{definition}

\newtheorem{definition}[theorem]{Definition}
\newtheorem{remark}{Remark}
\newtheorem{example}{Example}

\theoremstyle{plain}

\newcommand{\ignore}[1]{}

\newcommand{\dist}{\mathsf{dist}}

\newcommand{\Cov}[2]{\mathrm{Cov} \left[ #1, #2 \right]}
\newcommand{\Covs}[3]{\mathrm{Cov}_{#1}\left[ #2, #3 \right]}
\newcommand{\Var}[1]{\mathrm{Var} \left[ #1 \right]}

\newcommand{\Vars}[2]{{\mathrm{Var}}_{#1} \left[ #2 \right]}

\newcommand{\Ex}[1]{\bE \left[ #1 \right]}
\newcommand{\Exu}[2]{\underset{#1} \bE \left[ #2 \right] }

\renewcommand{\Pr}[1]{\bP \left[ #1 \right]} 

\newcommand{\define}{\vcentcolon=}

\newcommand{\floor}[1]{\ensuremath{\lfloor #1 \rfloor}}

\DeclarePairedDelimiter{\abs}{\lvert}{\rvert}

\newcommand{\ind}[1]{\mathds{1} \left[ #1 \right] }

\newcommand{\zo}{\{0,1\}}

\newcommand{\R}{\,\mathsf{R}}

\newcommand{\cE}{\ensuremath{\mathcal{E}}}

\newcommand{\bE}{\ensuremath{\mathbb{E}}}

\newcommand{\bN}{\ensuremath{\mathbb{N}}}
\newcommand{\bP}{\ensuremath{\mathbb{P}}}
\newcommand{\bR}{\ensuremath{\mathbb{R}}}

\usepackage[style=alphabetic,backend=biber,maxnames=99,maxalphanames=99]{biblatex}
\allowdisplaybreaks
\newcommand{\inp}[2]{\left\langle #1, #2 \right\rangle}
\newcommand{\grad}{\nabla}
\newcommand{\mono}{\mathsf{mono}}

\newcommand{\mix}{\mathsf{mix}}

\newcommand{\lap}{\mathscr{L}}
\newcommand{\dir}{\mathcal{E}} 

\DeclareMathOperator{\vol}{vol}

\renewcommand{\epsilon}{\varepsilon}

\newtoggle{anonymous}
\togglefalse{anonymous}

\title{On the Spectral Expansion of Monotone Subsets of the Hypercube}

\newcommand{\email}[1]{\href{mailto:#1}{\texttt{#1}}}

\iftoggle{anonymous}{%
\author{}
}{%
\author{Yumou Fei\thanks{Department of EECS, Massachusetts Institute of Technology. \email{yumou415@mit.edu}}
\and Renato Ferreira Pinto Jr.\thanks{University of Waterloo. \email{renato.ferreira@uwaterloo.ca}}}
}
\date{}

\begin{document}
\maketitle
\begin{abstract}
We study the spectral gap of subgraphs of the hypercube induced by monotone subsets of vertices. For a monotone subset $A\subseteq\{0,1\}^{n}$ of density $\mu(A)$, the previous best lower bound on the spectral gap, due to Cohen \cite{cohen2016problems}, was $\gamma\gtrsim \mu(A)/n^{2}$, improving upon the earlier bound $\gamma\gtrsim \mu(A)^{2}/n^{2}$ established by Ding and Mossel \cite{ding2014mixing}. In this paper, we prove the optimal lower bound $\gamma\gtrsim \mu(A)/n$. As a corollary, we improve the mixing time upper bound of the random walk on constant-density monotone sets from $O(n^{3})$, as shown by Ding and Mossel, to $O(n^{2})$. Along the way, we develop two new inequalities that may be of independent interest: (1)~a directed $L^{2}$-Poincar\'{e} inequality on the hypercube, and (2)~an ``approximate'' FKG inequality for monotone sets.
\end{abstract}

\thispagestyle{empty}
\setcounter{page}{0}
\newpage
{
\setcounter{tocdepth}{2} 
\tableofcontents
}
\thispagestyle{empty}
\setcounter{page}{0}
\newpage
\setcounter{page}{1}

\section{Introduction}

Suppose $G = (V, E)$ is a good expander graph, so that a random walk on the vertices of $G$ is
\emph{fast mixing}, \ie converges quickly to its stationary distribution. When is a random walk on a
subgraph of $G$ also fast mixing? More precisely, what kinds of subgraph restrictions preserve good
expansion?

This paper studies the case of the hypercube graph $H_n$, where vertices $x, y \in \zo^n$ are
connected by an edge if and only if they differ in exactly one coordinate. Recall that the
\emph{lazy random walk} on $H_n$ has mixing time $\Theta(n \log n)$. Given a subset $A
\subseteq \zo^n$ of vertices, we consider the random walk on $\zo^n$
\emph{censored} to $A$.

\begin{definition}[Censored random walk, \cite{ding2014mixing}]
    Given $A \subseteq \zo^n$, the \emph{random walk on $\zo^n$ censored to}
    $A$ is defined as follows. On state $x \in A$, sample $i \in [n]$ uniformly
    at random and let $y$ be the vertex obtained by flipping the $i$-th bit of
    $x$. Then
    \begin{enumerate}
        \item If $y \in A$, flip a coin and either stay at $x$ or move to $y$ (each with probability $1/2$).
        \item If $y \not\in A$, stay at $x$ (in which case we call this a \emph{censored} step).
    \end{enumerate}
\end{definition}

Without further guarantees on $A$, the censored random walk may mix well or
extremely poorly even when $A$ is large and connected, as the following two
examples illustrate:

\begin{example}[Subcube]
    Let $S \subset [n]$ be a set of indices, and let $A$ be the subcube given by
    the vertices $x \in \zo^n$ satisfying $x_i = 0$ for all $i \in S$. Then the
    censored random walk is essentially a random walk on the smaller cube
    $\zo^{n'}$, where $n' \define n - |S|$, except that only an
    $O(n'/n)$-fraction of the transitions are not censored. Thus the
    censored random walk has mixing time $O(\tfrac{n}{n'} \cdot n' \log n')
    = O(n \log n)$.
\end{example}

\begin{example}[Middle slice bridge]\label{ex:middle-slice}
    Let $x^* \in \zo^n$ be an arbitrary vertex with Hamming weight $|x^*| =
    \floor{n/2}$, and consider the set $A \define \left\{ x \in \zo^n : |x| \ne
    \floor{n/2} \right\} \cup \{ x^* \}$. A spectral argument shows that the
    mixing time of the censored random walk is exponential in $n$.
\end{example}

Thus, it is natural to ask: what properties of $A$ ensure fast mixing? In
\cite{ding2014mixing}, Ding \& Mossel initiated the study of random walks
censored to \emph{monotone} sets $A$\footnote{A set $A$ is called monotone if $x
\in A$ implies $y \in A$ whenever $x \preceq y$, where the latter denotes the
natural partial order on the hypercube: $x \preceq y$ if $x_i \le y_i$ for every
$i \in [n]$.} and showed that, when $A$ is not too small, monotonicity implies
fast mixing. Concretely, letting $\mu$ denote the uniform distribution on
$\zo^n$, they proved
\begin{theorem}[{\cite[Corollary 1.2]{ding2014mixing}}]
    \label{thm:ding-mossel}
    Let $A \subseteq \zo^n$ be a non-empty monotone set. Then the random walk on
    $\zo^n$ censored to $A$ has mixing time
    \[
        t_{\mix}\leq 512\cdot 
        \left(\frac{n}{\mu(A)}\right)^{2}\log(4\cdot 2^{n}\mu(A)).    
    \]
\end{theorem}

When the density $\mu(A)$ is a constant, the above implies a mixing time bound of $O(n^3)$. In particular, for the uncensored special case $A = H_n$, this result only yields an upper
bound of $O(n^3)$ on the mixing time, versus the optimal $\Theta(n \log n)$; this suggests the
potential for improving upon \cref{thm:ding-mossel}, and indeed \cite{ding2014mixing} asked the following question.

\begin{question}[{\cite[Question 1.1]{ding2014mixing}}]\label{q:n-logn-mixing}
Suppose $\mu(A)\geq \varepsilon$ for some constant $\varepsilon>0$. Is it true that $t_{\mix}\leq O_{\varepsilon}(n\log n)$?
\end{question}

Our main result makes progress on this question by showing an $O(n^{2})$ mixing time bound for monotone sets $A$ of constant density.

\begin{theorem}
    \label{thm:intro-mixing-time}
    Let $A \subseteq \zo^n$ be a non-empty monotone set. Then the random walk on
    $\zo^n$ censored to $A$ has mixing time
    \[
        t_{\mix} \leq \frac{2n}{\mu(A)}\cdot  \log(4\cdot 2^n \mu(A)) \,.
    \]
\end{theorem}

\subsection{Spectral gap}

It is well-known that the mixing time of a Markov chain is related to the \emph{spectral gap} of its generator (see e.g. \cite{levin2017markov}), or equivalently the \emph{spectral expansion} of the underlying graph. The Poincar\'{e} inequality (see e.g. \cite{o2014analysis}) for the hypercube states that the spectral expansion of the (lazy) hypercube $H_{n}$ is exactly $1/n$, which implies the (non-tight) mixing time bound $O(n^{2})$ for the lazy random walk on $H_{n}$.

Thus, a natural question related to mixing under censoring is the \emph{robustness} of the classical Poincaré inequality under vertex removal from $H_{n}$. Specifically, does the spectral expansion of $H_{n}$ remain on the order of $1/n$ if only a small fraction of the vertices are removed, or does it exhibit a significant deviation? \Cref{ex:middle-slice} demonstrates that the spectral expansion can shrink to exponentially small values if the removed set is arbitrary. Our goal is to show that when the removed set of vertices is monotone (and not too large), the spectral expansion remains at least on the order of $1/n$.

For the purpose of clearer comparison with the classical Poincar\'{e} inequality, we introduce the Dirichlet form of a function on $A$. Note that in the case $A=H_{n}$, the following definition is exactly the ``influence'' \cite[Definition 2.27]{o2014analysis} of the function $f:\{0,1\}^{n}\rightarrow\bR$. 

\begin{definition}
Fix a monotone set $A\subseteq \{0,1\}^{n}$. For all $f:A\rightarrow\bR$, we define\footnote{
The energy functional $\cE_{A}(\cdot)$ defined here differs from the standard Dirichlet form of the censored random walk on $A$ by a factor of $n$. This normalization is chosen to better align with the notion of total influence in the analysis of Boolean functions.
}
\[
\cE_{A}(f)\define \frac{1}{4}\cdot \Exu{x\in A}{\sum_{i=1}^{n}\left(f(x)-f(x^{\oplus i})\right)^{2}\cdot\ind{x^{\oplus i}\in A}}.
\]
Here $x^{\oplus i}$ denotes the binary string obtained by flipping the $i$-th bit of $x$.
\end{definition}

We can now state our ``robust'' version of the Poincar\'{e} inequality.

\begin{theorem}\label{thm:main}
Let $A\subseteq \{0,1\}^{n}$ be a non-empty monotone set. We have for all $f:A\rightarrow\bR$
\[
\Vars{A}{f}\leq\frac{1}{1-\sqrt{1-\mu(A)}}\cdot\cE_{A}(f).
\]
Here $\Vars{A}{f}$ stands for the variance of $f(x)$ where $x$ is a uniformly random element of $A$.
\end{theorem}
Note that in the case $A=H_{n}$, the above theorem recovers the Poincar\'{e} inequality on the
hypercube. We remark that \Cref{thm:intro-mixing-time} follows directly from \Cref{thm:main} due to
standard Markov chain theory (e.g. \cite[Theorem 12.4]{levin2017markov}), so the rest of the paper focuses mainly on proving \Cref{thm:main}. 

\Cref{thm:main} can also be stated as a lower bound on the spectral gap of $H_{A}$\textemdash the
subgraph of $H_n$ induced by $A$, with a self-loop added to vertex $x$ for each edge $\{x,y\}$ of the hypercube with
$x \in A$ and $y \not\in A$ (which counts as $1$ toward the degree of $x$). For convenience of reference, in this paper we \emph{define} the spectral gap using the language of \Cref{thm:main}.

\begin{definition}
Let $A\subseteq \zo^{n}$ be a monotone set with at least 2 elements. We define
$$\gamma(H_{A})\define \frac{1}{n}\cdot\inf_{f\not\in \mathsf{const}_{A}}\frac{\cE_{A}(f)}{\Vars{A}{f}},$$
where $f$ ranges over all non-constant functions from $A$ to $\bR$.
\end{definition}

Now \Cref{thm:main} can be stated as $\gamma(H_{A})\geq \frac{1}{n}\left(1-\sqrt{1-\mu(A)}\right)\gtrsim \mu(A)/n$, for monotone sets $A$ with $|A|\geq 2$.
\subsection{Proof overview: previous work}
\label{subsec:overview-previous}

We begin by briefly describing the proof of \Cref{thm:ding-mossel} in \cite{ding2014mixing}. The proof in \cite{ding2014mixing} also analyzes the spectral expansion of $H_{A}$, achieving the lower bound $\gamma(H_{A})\gtrsim \mu(A)^{2}/n^{2}$. 

By Cheeger's inequality, to obtain a lower bound on $\gamma(H_{A})$, it suffices to lower bound the \emph{bottleneck ratio}
$$\phi(H_{A})\define \min_{S\subseteq A}\frac{|E(S,A\setminus S)|}{\min\{|S|,|A\setminus S|\}},$$
where $E(S,A\setminus S)$ denotes the set of edges $\{x,y\}$ of $H_{n}$ with $x\in S$ and $y\in A\setminus S$. By an isoperimetric inequality of the hypercube, there is good lower bound on the number of boundary edges connecting a vertex in $S$ to a vertex in $\{0,1\}^{n}\setminus S$. However, it is not immediately clear how many of these edges actually lead to $A\setminus S$. The crucial observation of \cite{ding2014mixing} is that if an edge from a vertex $x\in S$ goes ``upward''\textemdash that is, if its other endpoint $y\in \{0,1\}^{n}\setminus S$ satisfies $y\succeq x$\textemdash then by the monotonicity of $A$, we must have $y\in A\setminus S$.

Coincidentally, there is a ``directed isopetrimetric inequality'' \cite[Theorem 2]{GGLRS00}, developed by the property testing community, which provides a lower bound on exactly the number of such \emph{upward} boundary edges. Specifically, it gives a lower bound on the number of edges connecting a vertex $x\in S$ to a vertex $y\in \{0,1\}^{n}\setminus S$ with $y\succeq x$ in terms of the \emph{distance} of $S$ to monotonicity. The precise notion of distance is less important than the key fact that, when $|A|$ is not too small, at least one of $S$ and $A\setminus S$ must be far from any monotone set\textemdash i.e., it must have a large distance to monotonicity\textemdash due to the FKG inequality \cite{fortuin1971correlation}. As a result, we can lower bound the number of upward boundary edges from either $S$ or $A\setminus S$. Since both sets of upward edges are subsets $E(S,A\setminus S)$, we can thus arrive at a lower bound on $|E(S,A\setminus S)|$ and hence $\phi(H_{A})$.

The work \cite{ding2014mixing} actually achieves the optimal bound $\phi(H_{A})\gtrsim \mu(A)/n$ on the bottleneck ratio. However, this only translates to a quadratically worse bound $\gamma(H_{A})\gtrsim \mu(A)^{2}/n^{2}$ for the spectral expansion, due to the loss incurred by applying Cheeger's inequality. One natural idea is to avoid using Cheeger's inequality by directly bounding the
spectral gap $\gamma(H_A)$. In this direction, \cite{cohen2016problems} used a
canonical path argument to show the bound $\gamma(H_{A})\gtrsim \mu(A)/n^{2}$, 
which improves upon \cref{thm:ding-mossel} by a factor of $\mu(A)$. However,
this improvement is only effective when $\mu(A) \ll 1$ and the dependence on $n$ remains suboptimal.

\subsection{Proof overview: our work}

At first glance, the proof in \cite{ding2014mixing}, as described in the previous subsection,
appears to heavily depend on the discrete nature of the bottleneck ratio. In our view, a key
conceptual contribution of this work is that the arguments in \cite{ding2014mixing} can be adapted
to the $L^{2}$ setting. While the discrete setting leads to the bottleneck ratio, in the $L^{2}$
setting, the corresponding arguments directly lead to the spectral expansion as stated in \Cref{thm:main}. For a full set of analogies, see \Cref{tab:analogy}. 

\begin{table}[h!]
\centering
\begin{tabular}{|c|c|}
\hline
\rowcolor{gray!30} \textbf{The discrete setting} & \textbf{The $L^{2}$ setting} \\ \hline
subset $S\subseteq A$ & function $f:A\rightarrow\mathbb{R}$ \\ \hline
the complement set $A\setminus S$ & the function $-f$ \\ \hline
$|E(S,A\setminus S)|$ & $\cE_{A}(f)$ \\ \hline
$\min\{|S|,|A\setminus S|\}$ & $\Vars{A}{f}$ \\ \hline
bottleneck ratio $\phi(H_{A})$ & spectral gap $\gamma(H_{A})$ \\ \hline
directed isoperimetric inequality \cite{GGLRS00} & directed $L^{2}$-Poincar\'{e} inequality (\Cref{thm:directed})\\ \hline
classical FKG inequality \cite{fortuin1971correlation} & approximate FKG inequality (\Cref{thm:approximate-FKG}) \\ \hline
\end{tabular}
\caption{The analogies between the discrete and $L^{2}$ settings}\label{tab:analogy}
\end{table}

In contrast to \cite{ding2014mixing}, where the two main inequalities used in the proof\textemdash the directed isoperimetric inequality from \cite{GGLRS00} and the FKG inequality from \cite{fortuin1971correlation}\textemdash are classical results, in our $L^{2}$ settings we have to formulate and prove new versions of these inequalities, which may be of interest on their own.

\paragraph*{Directed Poincaré inequality.}

As indicated in \Cref{subsec:overview-previous}, directed isoperimetric
inequalities aim to lower bound the number of ``upward boundary edges'' from a
set $S$ to $\{0,1\}^{n}\setminus S$ in terms of the ``distance'' of $S$ to
monotonicity; see \cref{sec:related-work} for more background.

For our application, we require a directed isoperimetric inequality in the $L^2$
setting, which is the setting associated with the classical Poincaré inequality
and the spectral gap. The first step is to define, for any
$f:\{0,1\}^{n}\rightarrow\mathbb{R}$, its $L^{2}$-distance to monotonicity and
its ``upward boundary edges''.

\begin{definition}[Distance to monotonicity]
For a function $f:\zo^{n}\rightarrow\bR$, we define
\[
\dist_{2}^{\mono}(f)\define \inf_{g\in \mono}\sqrt{\Exu{x\in\{0,1\}^{n}}{\big(f(x)-g(x)\big)^{2}}},
\]
where $g$ ranges over all monotone increasing functions from $\{0,1\}^{n}$ to $\bR$.
\end{definition}

\begin{definition}[Upward boundary]\label{def:upward-boundary}
For all $f:\{0,1\}^{n}\rightarrow\bR$, we define
\[
\cE^{-}(f)\define \frac{1}{4}\cdot \Exu{x\in \{0,1\}^{n}}{\sum_{i=1}^{n}\min\left\{0,f(x^{i\rightarrow 1})-f(x^{i\rightarrow 0})\right\}^{2}}.
\]
Here for $x\in\zo^{n}$ and $b\in \zo$, $x^{i\rightarrow b}$ stands for the string $(x_{1},\dots,x_{i-1},b,x_{i+1},\dots,x_{n})$.
\end{definition}

We are now ready to state our directed $L^{2}$-Poincar\'{e} inequality.
\begin{theorem}[Directed Poincar\'{e} inequality]\label{thm:directed}
For all functions $f:\{0,1\}^{n}\rightarrow\bR$, we have 
\[
\dist_{2}^{\mono}(f)^{2}\leq \cE^{-}(f).
\]
\end{theorem}

\paragraph*{Approximate FKG inequality.}

The classical FKG inequality of \cite{fortuin1971correlation} states that if $f,g:\zo^{n}\rightarrow\bR$ are monotone increasing functions and $x$ is a uniformly random element of $\zo^{n}$, then the random variables $f(x)$ and $g(x)$ are nonnegatively correlated.  It is well-known that this statement holds for increasing functions over a broader class of partially ordered sets (posets). In our proof, we crucially need a lower bound on the correlation ratio of any two increasing functions $A\rightarrow\bR$, where the set $A$ is partially ordered by the natural partial order of the hypercube. However, it is easy to see that the FKG inequality does not generally hold on this poset. Thus, we seek an ``approximate'' version of the FKG inequality, where we are content with a correlation ratio bounded away from $-1$, rather than necessarily nonnegative.

\begin{definition}[Approximate FKG ratio]\label{def:FKG-ratio}
Fix a monotone set $A\subseteq\zo^{n}$ with at least 2 elements. We define the \emph{approximate FKG ratio} of the poset $A$ to be
$$\delta(A):=\min\left\{0,\inf_{f,g\in \mono _{A}\setminus \mathsf{const}_{A}}\frac{\Covs{A}{f}{g}}{\sqrt{\Vars{A}{f}\cdot\Vars{A}{g}}}\right\},$$
where $f$ and $g$ range over all non-constant monotone increasing functions from $A$ to $\bR$. Here, $\Covs{A}{f}{g}$ stands for the covariance of the random variable pair $(f(x),g(x))$ where $x$ is a uniformly random element of $A$.
\end{definition}

\begin{theorem}[Approximate FKG inequality]\label{thm:approximate-FKG}
For any monotone set $A\subseteq\{0,1\}^{n}$ with at least 2 elements, we have $\delta(A)\geq -\sqrt{1-\mu(A)}$.
\end{theorem}
\subsection{The case of small $A$: fast mixing requires good FKG ratio}

The bound in \Cref{thm:main} gives only a spectral gap bound $\gamma(H_{A})\gtrsim \mu(A)/n$ for the random walk censored to $A$. The dependence on $n$ is clearly optimal: even in the case $A=H_{n}$, the spectral gap is exactly $1/n$. The next example shows that the asymptotic dependence on $\mu(A)$ is also close to optimal.

\begin{example}[{\cite[Example 1.3]{ding2014mixing}}]\label{ex:torpid-mixing}
Assume $n/4\leq m\leq n/2$ and consider $A=\{x\in\zo^{n}:x_{1}=\dots=x_{m}=1\}\cup\{x\in\zo^{n}:x_{m+1}=\dots=x_{2m}=1\}$, the union of two subcubes. In this case, $\mu(A)\sim 2^{-m+1}$. Let $f:A\rightarrow\bR$ be defined by $f(x)=1$ if $x_{1}=\dots=x_{m}=1$ and $f(x)=-1$ otherwise. Then $\cE_{A}(f)\sim m\cdot 2^{-m}$ and $\Vars{A}{f}\sim 1$, so $\gamma(H_{A})\lesssim \mu(A)$. Standard Markov chain theory (e.g. \cite[Theorem 7.4]{levin2017markov}) shows that the mixing time of the random walk censored to $A$ is exponentially large in $n$.
\end{example}

Remarkably, \Cref{ex:torpid-mixing} is also where the approximate FKG inequality fails badly\textemdash if we let $A$ be the union of two subcubes as in \Cref{ex:torpid-mixing} and consider the indicator functions of the two subcubes, it is easy to see that they are both increasing functions on the poset $A$ but are very anti-correlated (\ie the approximate FKG ratio $\delta(A)$ is very close to $-1$).

Our results actually reveal that, when $A$ is a monotone set, torpid mixing happens \emph{if and only if} the approximate FKG ratio of $A$ is close to $-1$.

\begin{theorem}\label{thm:FKG-to-spectral-gap}
Let $A\subseteq\{0,1\}^{n}$ be a monotone set with at least 2 elements. Then for all functions $f:A\rightarrow\bR$, we have
\[
(1+\delta(A))\cdot \Vars{A}{f}\leq \cE_{A}(f).
\]
\end{theorem}

\begin{theorem}\label{thm:spectral-gap-to-FKG}
Let $A\subseteq\{0,1\}^{n}$ be a monotone set with at least 2 elements. Then for some non-constant function $f:A\rightarrow\bR$, we have
\[
(1+\delta(A))\cdot n\cdot \Vars{A}{f}\geq \cE_{A}(f).
\]
\end{theorem}

The two theorems above imply that $(1+\delta(A))/n\leq \gamma(H_{A})\leq 1+\delta(A)$, which means the approximate FKG ratio of $A$ characterizes the spectral gap of $H_{A}$ up to a factor of $n$.

\begin{remark}
The case $A=H_{n}$ demonstrates that the lower bound $(1+\delta(A))/n\leq \gamma(H_{A})$ is tight. Moreover, there are examples indicating that the upper bound $\gamma(H_{A})\leq 1+\delta(A)$ is tight up to a constant factor\textemdash for instance, when $n\geq 2$ and $A=\{x\in \zo^{n}:|x|\geq n-1\}$, we have $\gamma(H_{A})=\frac{1}{2n}$ and $1+\delta(A)\leq \frac{2}{n}$.
\end{remark}

\subsection{Open problems}

Our work leaves two avenues for potential improvement, roughly corresponding to
two regimes in the size of the set $A$. We discuss each of these directions in
turn.

\paragraph*{Large $A$.} When $\mu(A)\geq \varepsilon$ for some fixed constant $\varepsilon>0$, we establish the tight asymptotic bound $\gamma(H_{A})\gtrsim 1/n$. However, this only yields the mixing time bound $t_{\mix}=O_{\varepsilon}(n^{2})$, which does not resolve \Cref{q:n-logn-mixing}.
One way to establish an $O(n\log n)$ mixing time bound would be to prove a
log-Sobolev inequality instead of an $L^{2}$-Poincar\'{e} inequality.
It is plausible that our techniques could be further adapted to establish a log-Sobolev inequality, similar to how we extended the argument of \cite{ding2014mixing} from the discrete setting to the $L^{2}$ setting. Are there analogous versions of the directed isoperimetric inequality and the approximate FKG inequality in the log-Sobolev setting? We leave these as open questions.

\paragraph*{Small $A$.} When no additional structure is imposed, the random walk censored to $A$ may mix very slowly if $\mu(A)\ll 1$ (\eg \Cref{ex:torpid-mixing}). In many problems of interest, however, $A$ possesses some form of structure, and the goal is to obtain a good bound on the mixing time, aiming for an efficient approximate sampling algorithm. For example, when $A$ is a halfspace, \ie defined by $A=\{x\in\zo^{n}:a_{1}x_{1}+\dots+a_{n}x_{n}\geq b\}$ for nonnegative numbers $a_{1},\dots,a_{n},b$, \cite{morris2004random} proves a mixing time bound of $n^{9/2+o(1)}$ which yields a sampling algorithm for 0-1 knapsack solutions.

Our work (\Cref{thm:FKG-to-spectral-gap,thm:spectral-gap-to-FKG}) reveals that the deciding factor for whether rapid mixing holds is the \emph{approximate FKG ratio} of $A$ rather than the size of $A$. However, we do not know how to leverage additional structure of $A$ in a direct study of its approximate FKG ratio, and we leave the development of new tools for this purpose as an interesting direction for future work.

\subsection{Related work}
\label{sec:related-work}

\paragraph*{Mixing time of censored random walks.} Markov Chain Monte Carlo methods are the object
of extensive study in mathematics, statistical physics and theoretical computer science, and the
question of mixing time of random walks lies at the core of algorithms for approximate sampling and
counting; see \eg \cite{Jer98,Gur16,MT06,levin2017markov}. In settings featuring combinatorial
structure such as in sampling matchings, independent sets, or spanning forests of a graph, or their
natural generalizations to the more algebraic setting of matroids, a \emph{basis exchange} or
\emph{down-up} random walk is usually employed, and spectral arguments are used to bound the mixing
time; see \eg \cite{jerrum2003counting,ALGV19,CGM19,AL20}.

Our work focuses on the setting where the set $A$ may not enjoy such rich
structure, and instead is only guaranteed to be monotone. As discussed above,
our results improve upon the spectral gap and mixing time bounds shown by the
previous works of \cite{ding2014mixing,cohen2016problems}. In the special case
where the monotone set $A$ is additionally promised to contain every $x \in
\zo^n$ with Hamming weight at least $(\tfrac{1}{2}-\epsilon)n$, \ie the middle
layers of the hypercube, one may expect the censored and uncensored random walks
to behave similarly, and indeed in this case \cite{Mat02} gave the optimal
$\Omega_\epsilon(1/n)$ bound for the spectral gap and log-Sobolev constant of
the censored random walk, which implies the optimal $O_\epsilon(n \log n)$
mixing time bound.

Besides combinatorial or algebraic structure, one may also ask what
\emph{geometric} structure affords fast mixing. As mentioned above,
\cite{morris2004random} studied the censored random walk when the set $A$ is a
halfspace, \ie $A=\{x\in\zo^{n}:a_{1}x_{1}+\dots+a_{n}x_{n}\leq b\}$ for
nonnegative numbers $a_{1},\dots,a_{n},b$, which corresponds to 0-1 knapsack
solutions. Intuitively, such set $A$ should not contain bottlenecks even if it
is very small, and indeed \cite{morris2004random} showed a mixing time bound of
$n^{9/2+o(1)}$. By our \cref{thm:spectral-gap-to-FKG}, this also gives an
inverse polynomial lower bound on the quantity $1+\delta(A)$. On the other hand, a lower bound of
$\widetilde \Omega(n^2)$ holds for the mixing time \cite{morris2004random}, and closing this gap is
an interesting open problem\footnote{In a different line of investigation, a series of works has
explored approaches for approximately sampling and counting knapsack solutions using dynamic
programming \cite{Dye03,GKMSVV11,SVV12,GMW18,RT19,FJ25};
the best results in this direction are $\widetilde O(n^{5/2})$ and
$\widetilde O(n^4)$-time algorithms for approximately sampling (depending on the model of
computation) \cite{Dye03,GMW18,RT19}, and a (subquadratic) $\widetilde O(n^{3/2})$-time algorithm
for approximately counting knapsack solutions \cite{FJ25}. However, these results do not directly
say anything about the mixing time of the random walk on knapsack solutions.}\!\!.

\paragraph*{Directed isoperimetric inequalities.} Directed versions of isoperimetric inequalities
such as the classical Poincaré inequality and Talagrand's inequality \cite{Tal93} have emerged over
the last couple of decades as a key tool in the field of property testing. Since the introduction of
the problem of monotonicity testing \cite{GGLRS00} and especially since the work of \cite{CS16},
directed isoperimetric inequalities have unlocked new results on the query complexity of testing
monotonicity of Boolean functions over discrete domains such as the hypercube and hypergrid
\cite{KMS18,BCS18,PRW22,BCS23a,BCS23b,BKKM23},
and, more recently, real-valued functions over the hypercube \cite{BKR24} and the
continuous cube \cite{pinto2024directed}. In recent work, \cite{CCRSW25} proved such an inequality
toward testing monotonicity of \emph{probability distributions} over the hypercube, using certain
conditional samples from that distribution. To the best of our knowledge, the work of
\cite{ding2014mixing} was the first application of a directed isoperimetric inequality (\ie the
inequality of \cite{GGLRS00} for Boolean functions) outside of property testing, and the present
work seems to be the first application of a real-valued directed isoperimetric inequality outside of
property testing.

The specific Poincaré inequality we prove in \cref{thm:directed} is most closely related to the
following previous developments. The work of \cite{Fer23} introduced the systematic study of
directed $L^p$-Poincaré inequalities for monotonicity testing, and proved an $L^1$ version of our
\cref{thm:directed}; in the language of that paper, \cref{thm:directed} is a directed $(L^2,
\ell^2)$-Poincaré inequality, whereas \cite{Fer23} proved an $(L^1, \ell^1)$ inequality. Another
similar inequality for real-valued functions was proved by \cite{BKR24}, who considered the Hamming
distance as opposed to $L^p$ distance. An $(L^2, \ell^2)$ inequality (\ie the same flavor as ours)
was proved for functions over the \emph{continuous} cube $[0,1]^n$ in \cite{pinto2024directed}, and
our proof via the study of a dynamical system is directly inspired by theirs. Most recently,
\cite{CCRSW25} proved a directed $(L^1, \ell^2)$-Poincaré inequality\footnote{Compared to
\cref{thm:directed}, that inequality uses the $L^1$ as opposed to $L^2$ distance, and takes a
square-root inside the expectation operator in our definition of $\cE^-(f)$ in
\cref{def:upward-boundary}. By Jensen's inequality, the square root of the left- and right-hand
sides of our inequality are respectively larger than the left- and right-hand sides of the
inequality of \cite{CCRSW25}, so the two results are not immediately comparable.} for real-valued
functions on the hypercube, by extending the result of \cite{KMS18} for Boolean functions via a
thresholding argument of \cite{BRY14}. There does not seem to be a trivial reduction between our
inequality and the foregoing results.

\subsection{Organization of the paper}

It is clear that \Cref{thm:approximate-FKG,thm:FKG-to-spectral-gap} together imply \Cref{thm:main},
and we recall that standard Markov chain theory (\cite[Theorem 12.4]{levin2017markov}) derives \Cref{thm:intro-mixing-time} from \Cref{thm:main}. 

\Cref{sec:approximate-FKG} presents a proof of the approximate FKG inequality for large monotone sets (\Cref{thm:approximate-FKG}). \Cref{sec:directed-to-undirected} is where the heart of the argument of \cite{ding2014mixing} is carried out in the $L^{2}$ setting. \Cref{sec:directed-to-undirected} demonstrates that the directed isoperimetric inequality of $\zo^{n}$ (\Cref{thm:directed}) implies the undirected isoperimetric inequality of the monotone subset $A$ (\Cref{thm:FKG-to-spectral-gap}), and why the approximate FKG ratio of $A$ is important for this implication. \Cref{sec:spectral-directed,sec:directed-gap-of-hypercube} are devoted to proving the directed $L^{2}$-Poincar\'{e} inequality (\Cref{thm:directed}). We illustrate this organization using the following diagram.

\begin{center}
\begin{tikzpicture}

\node (A) [hidden] at (0,0) {};
\node (B) [node, right=3cm of A] {\Cref{thm:directed}};
\node (C) [node, right=2cm of B] {\Cref{thm:FKG-to-spectral-gap}};
\node (E) [node, above=1.5cm of C] {\Cref{thm:approximate-FKG}};
\node (D) [hidden, left=2cm of E] {}; 
\node (F) [node, right=0.5cm of C, yshift=1.25cm] {\Cref{thm:main}};
\node (G) [node, right=0.5cm of F] {\Cref{thm:intro-mixing-time}};

\draw [double arrow] (A) -- (B) node[midway, above] {\Cref{sec:spectral-directed,sec:directed-gap-of-hypercube}};  
\draw [double arrow] (B) -- (C) node[midway, above] {\Cref{sec:directed-to-undirected}}; 
\draw [double arrow] (D) -- (E) node[midway, above] {\Cref{sec:approximate-FKG}};  
\draw [double arrow] (C) -- (F);
\draw [double arrow] (E) -- (F);
\draw [double arrow] (F) -- (G);

\end{tikzpicture}
\end{center}

Finally, \Cref{sec:reverse} contains a proof of the \Cref{thm:spectral-gap-to-FKG}, which is logically independent from the proof of the main result \Cref{thm:intro-mixing-time}.

The following notational convention is used throughout the paper.

\paragraph{Notation.} Given a discrete set $V$, we denote by $L^2(V)$ the Hilbert space obtained by endowing the set of $V
\to \bR$ functions with the inner product $\inp{f}{g} \define \Exu{x \in V}{f(x) g(x)}$, where the
expectation is taken with respect to the uniform distribution over $V$. This inner product induces
the norm $\|f\|_{2} = \sqrt{\inp{f}{f}} = \sqrt{\Ex{f^2}}$.

\section{Approximate FKG inequality}\label{sec:approximate-FKG}

The goal of this section is to prove \Cref{thm:approximate-FKG}, where we need to lower bound the correlation ratio between two monotone increasing functions on $A$. We first note that the case where the functions take values in $\{0,1\}$ is easy. Indeed, we have the following simple lemma.

\begin{lemma}\label{lem:FKG-zo}
Assume that $f,g:A\rightarrow\{0,1\}$ are monotone increasing functions. Then we have $\Exu{x\in A}{f(x)g(x)}\geq \mu(A)\cdot\Exu{x\in A}{f(x)}\cdot \Exu{x\in A}{g(x)}$.
\end{lemma}

\begin{proof}
Let $B=\{x\in A:f(x)=1\}$ and $C=\{x\in A:g(x)=1\}$. By the monotonicity of $f$ and $g$, the sets $B$ and $C$ are both monotone subsets of the hypercube $\zo^{n}$. By the classical FKG inequality \cite{fortuin1971correlation} we know that $\mu(B\cap C)\geq \mu(B)\cdot \mu(C)$. Therefore,
\[
\Exu{x\in A}{f(x)g(x)}=\frac{\mu(B\cap C)}{\mu(A)}\geq \mu(A)\cdot \frac{\mu(B)}{\mu(A)}\cdot\frac{\mu(C)}{\mu(A)}=\mu(A)\cdot\Exu{x\in A}{f(x)}\cdot \Exu{x\in A}{g(x)}. \qedhere
\]
\end{proof}

It is straightforward to deduce from \Cref{lem:FKG-zo} that for monotone increasing functions $f,g:A\rightarrow\{0,1\}$, the desired approximate FKG inequality 
$$\Covs{A}{f}{g}\geq -\sqrt{1-\mu(A)}\cdot\sqrt{\Vars{A}{f}\cdot\Vars{A}{g}}$$ holds. 

The main challenge in \Cref{thm:approximate-FKG} lies in extending this idea to real-valued functions. In fact, the problem can be reduced to proving the following statement, which involves purely random variables rather than any structural property of the partially ordered set $A$.

\begin{theorem}\label{thm:FKG}
Let $(X,Y)$ be a pair of real-valued random variables with bounded second moment. Suppose there is a constant $c\in[0,1)$ such that for all $a,b\in\mathbb{R}$,
\begin{equation}\label{eq:FKG-condition}
\Pr{X\geq a, Y\geq b}\geq c\cdot\Pr{X\geq a}\cdot\Pr{Y\geq b},
\end{equation}
then we must have
\begin{equation}\label{eq:FKG-conclusion}
\Cov{X}{Y}\geq -\sqrt{(1-c)\cdot\Var{X}\cdot\Var{Y}}.
\end{equation}
\end{theorem}

\begin{proof}[Proof of \Cref{thm:approximate-FKG} assuming \Cref{thm:FKG}]
Let $x$ be a uniformly random element of $A$ and let $X=f(x)$ and $Y=g(x)$. Thus $\Var{X}=\Vars{A}{f}$, $\Var{Y}=\Vars{A}{g}$, and $\Cov{X}{Y}=\Covs{A}{f}{g}$.

Now for each pair of $a,b\in\bR$, if we define $f_{a},g_{b}:A\rightarrow\bR$ by 
$$f_{a}(x)\define \ind{f(x)\geq a}\qquad\text{and}\qquad g_{b}(x)\define\ind{g(x)\geq b},$$
since they are clearly monotone increasing 0/1-valued functions, we can apply \Cref{lem:FKG-zo} to $f_{a}$ and $g_{b}$ to deduce that
\[\Pr{X\geq a, Y\geq b}\geq \mu(A)\cdot\Pr{X\geq a}\cdot\Pr{Y\geq b}.\]
If $\mu(A)=1$, then $A=\zo^{n}$ and the conclusion follows from the classical FKG inequality. If $\mu(A)<1$, we apply \Cref{thm:FKG} to the random variable pair $(X,Y)$ with constant $c=\mu(A)$, which yields exactly the desired conclusion. 
\end{proof}

The remainder of this section is devoted to proving \Cref{thm:FKG}, which turns out to be surprisingly nontrivial. To illustrate the complexity behind this inequality, we note that equality in \eqref{eq:FKG-conclusion} holds for a wide range of joint distributions of $(X,Y)$ beyond the case captured by \Cref{lem:FKG-zo}, i.e. where $X$ and $Y$ take only two possible values.

\begin{example}
Let $(X,Y)$ follow a discrete distribution supported on the grid $\{0,2,3\}^{2}$. Specifically, let 
$$\Pr{X=3,Y=3}=\Pr{X=3,Y=2}=\Pr{X=2,Y=3}=\frac{1}{5},$$
$$\Pr{X=0,Y=3}=\Pr{X=3,Y=0}=\frac{1}{15},\qquad \text{and}\qquad \Pr{X=2,Y=2}=\frac{4}{15}.$$
It is easy to check that \eqref{eq:FKG-condition} holds for $c=45/49$ and all $a,b\in \bR$. On the other hand, we have $\Cov{X}{Y}=-8/45$ and $\Var{X}=\Var{Y}=28/45$, so equality in \eqref{eq:FKG-conclusion} holds for $c=45/49$ as well.
\end{example}

\subsection{A symmetric model}

A key challenge in \Cref{thm:FKG} lies in its lack of ``centrosymmetry'' with respect to $(X,Y)$. While the assumption \eqref{eq:FKG-condition} does not remain invariant under the substitution $X \mapsto -X$ and $Y \mapsto -Y$, the conclusion is unaffected by such substitutions. This raises an intriguing question: what is the ``symmetric'' information inherent in \eqref{eq:FKG-condition} that leads to the conclusion?

In this subsection, we present an approach that effectively extracts the ``symmetric'' information from \eqref{eq:FKG-condition}. To this end, we first define two Borel measures on $[0,1]$ induced by $X$ and $Y$.
\begin{definition}\label{def:push-forward}
Let $\varphi_{X}:\mathbb{R}\rightarrow[0,1]$ be the Borel-measurable map defined by $a\mapsto \Pr{X\geq a}$, and then for each Borel set $E\subseteq [0,1]$, let $\alpha(E)$ be the Lebesgue measure of the inverse image $\varphi_{X}^{-1}(E)$. The countable additivity of $\alpha$ easily follows from the countable additivity of the Lebesgue measure.

The measure $\alpha$ is referred to as the push-forward of the Lebesgue measure on $\mathbb{R}$ by the map $a\mapsto \Pr{X\geq a}$. Similarly define the Borel measure $\beta$ on $[0,1]$ to be the push-forward of the Lebesgue measure by the map $b\mapsto \Pr{Y\geq b}$.
\end{definition}

The definition of push-forward measures naturally leads to the following ``change of variable'' formula, which is a standard fact in measure theory.

\begin{proposition}[{\cite[Theorem 3.6.1]{bogachev2007measure}}]\label{prop:change-of-variable}
    Suppose $\varphi$ is a Borel-measurable map from $\mathbb{R}$ to $[0,1]$, and suppose $\lambda$ is the push-forward of the Lebesgue measure under $\varphi$. Then for any Borel-measurable function $f:[0,1]\rightarrow\mathbb{R}^{\geq 0}$, we have
    $$\int_{0}^{1}f(x)\odif \lambda(x)=\int_{\mathbb{R}}f(\varphi(a))\odif a.$$
\end{proposition}
We then define the reverse of a measure, which corresponds to the substitution $X\mapsto -X$.

\begin{definition}
If $\lambda$ is a Borel measure on $[0,1]$, we let $\lambda^{\R}$ be the Borel measure on $[0,1]$ defined by $\lambda^{\R}(E)=\lambda(\{1-x:x\in E\})$, for Borel subsets $E$ of $[0,1]$.
\end{definition}

The following definition is the crucial tool in our proof of \Cref{thm:FKG}.

\begin{definition}
Fix a constant $c\in [0,1)$. We define the operator $K_{c}(\cdot,\cdot)$ by
$$K_{c}(\lambda,\nu):=\int\int \min\left\{\sqrt{1-c}\cdot xy,\frac{(1-x)(1-y)}{\sqrt{1-c}}\right\}\odif \lambda(x)\odif \nu(y),$$
for Borel measures $\lambda,\nu$ on $[0,1]$. For $c=0$, we omit the subscript and use the shorthand $K\define K_{0}$.
\end{definition}

The next two propositions demonstrate that the operator $K_{c}(\cdot,\cdot)$ is able to capture the variances and covariance of $X$ and $Y$. \Cref{prop:Cov-to-K} is the key place where ``symmetric'' information is extracted from the condition \eqref{eq:FKG-condition}. 

\begin{proposition}\label{prop:Var-to-K}
We have $\Var{X}=K(\alpha,\alpha^{\R})$. Similarly, $\Var{Y}=K(\beta,\beta^{\R})$.
\end{proposition}
\begin{proof}
Writing expected values as integrations of cumulative distribution functions (the ``layer cake representation''), we have 
\begin{align*}
\Var{X}&=\Ex{X^{2}}-\Ex{X}^{2}\\
&=\int_{-\infty}^{\infty}\int_{-\infty}^{\infty}\Big(\Pr{X\geq a,X\geq b}-\Pr{X\geq a}\Pr{X\geq b}\Big)\odif a\odif b\\
&=\int_{-\infty}^{\infty}\int_{-\infty}^{\infty}\min\Big\{\Pr{X\geq a}\Big(1-\Pr{X\geq b}\Big), \Pr{X\geq b}\Big(1-\Pr{X\geq a}\Big)\Big\}\odif a\odif b\\
&=\int_{0}^{1}\int_{0}^{1}\min\{x(1-y),y(1-x)\}\odif \alpha(x)\odif \alpha(y)\\
&=\int_{0}^{1}\int_{0}^{1}\min\{xy,(1-x)(1-y)\}\odif \alpha(x)\odif \alpha^{\R}(y)=K(\alpha,\alpha^{\R}),
\end{align*}
where the fourth equality above follows from \Cref{prop:change-of-variable}.
\end{proof}

\begin{proposition}\label{prop:Cov-to-K}
Assuming \eqref{eq:FKG-condition}, we have $\Cov{X}{Y}\geq -\sqrt{1-c}\cdot K_{c}(\alpha,\beta)$.
\end{proposition}

\begin{proof}
In a similar way to the proof of \Cref{prop:Var-to-K}, we have
\begin{align}
\Cov{X}{Y}&=\Ex{XY}-\Ex{X}\Ex{Y}\nonumber\\
&=\int_{-\infty}^{\infty}\int_{-\infty}^{\infty}\Big(\Pr{X\geq a,Y\geq b}-\Pr{X\geq a}\Pr{Y\geq b}\Big)\odif a\odif b.\label{eq:Cov-original}
\end{align}
Note that on one hand, by \eqref{eq:FKG-condition} we have
\begin{equation}\label{eq:Cov-bound-1}
\Pr{X\geq a,Y\geq b}-\Pr{X\geq a}\Pr{Y\geq b}\geq -(1-c)\Pr{X\geq a}\Pr{Y\geq b}.
\end{equation}
On the other hand, by union bound we have
\begin{align}
\Pr{X\geq a,Y\geq b}-\Pr{X\geq a}\Pr{Y\geq b}
&\geq 1-\Pr{X<a}-\Pr{Y<b}-\Pr{X\geq a}\Pr{Y\geq b}\nonumber\\
&=-\Big(1-\Pr{X\geq a}\Big)\Big(1-\Pr{Y\geq b}\Big).\label{eq:Cov-bound-2}
\end{align}
Plugging \eqref{eq:Cov-bound-1} and \eqref{eq:Cov-bound-2} into \eqref{eq:Cov-original}, we have
\begin{align*}
\Cov{X}{Y}&\geq
-\int_{-\infty}^{\infty}\int_{-\infty}^{\infty}\min\Big\{(1-c)\Pr{X\geq a}\Pr{Y\geq b},\Big(1-\Pr{X\geq a}\Big)\Big(1-\Pr{Y\geq b}\Big)\Big\}\odif a \odif b\\
&=-\int_{0}^{1}\int_{0}^{1}\min\{(1-c)xy,(1-x)(1-y)\}\odif \alpha(x) \odif \beta(y)\\
&=-\sqrt{1-c}\cdot K_{c}(\alpha,\beta),
\end{align*}
where the first equality above follows from \Cref{prop:change-of-variable}.
\end{proof}

We can now reduce \Cref{thm:FKG} to the following more ``symmetric'' lemma.

\begin{lemma}\label{lem:FKG-main}
For any two Borel measures $\alpha,\beta$, and any constant $c\in[0,1)$ we have
$$K_{c}(\alpha,\beta)^{2}\leq K(\alpha,\alpha^{\R})\cdot K(\beta,\beta^{\R}).$$
\end{lemma}

\begin{proof}[Proof of \Cref{thm:FKG} assuming \Cref{lem:FKG-main}]
Using \Cref{prop:Cov-to-K}, \Cref{lem:FKG-main} and \Cref{prop:Var-to-K} successively, we have
\begin{align*}
\Cov{X}{Y}&\geq -\sqrt{1-c}\cdot K_{c}(\alpha,\beta) \\
&\geq -\sqrt{1-c}\cdot\sqrt{K(\alpha,\alpha^{\R})K(\beta,\beta^{R})}\\
&=-\sqrt{(1-c)\cdot\Var{X}\cdot\Var{Y}}.\qedhere
\end{align*}
\end{proof}

The rest of this section is devoted to proving \Cref{lem:FKG-main}.

\subsection{Technical preparations}\label{subsec:FKG-technical}

Throughout this subsection, we let $c$ be a fixed constant in the range $[0,1)$. We will define two functions $h$ and $q$ that will be useful in our proof of \Cref{lem:FKG-main}.

\begin{definition}
We define a function $h:[0,1]\times[0,1]\rightarrow \mathbb{R}$ by
$$h(x,y):=-\left(\odv*{\frac{x}{\sqrt{1-cx}}}{x}\right)\left(\odv*{\frac{1-y}{\sqrt{1-cy}}}{y}\right).$$
\end{definition}

\begin{proposition}\label{prop:integration-h}
The function $h$ is continuous and nonnegative on $[0,1]\times[0,1]$ and for all $x,y\in[0,1]$,
$$\int_{0}^{x}\int_{y}^{1}h(r,s)\odif s\odif r=\frac{x(1-y)}{\sqrt{(1-cx)(1-cy)}}.$$
\end{proposition}

\begin{proof}
The integration of $h$ follows from definition, and straightforward calculation shows
\[
h(x,y)=\frac{(2-cx)(2-c-cy)}{4(1-cx)^{3/2}(1-cy)^{3/2}}\geq 0.\qedhere
\]
\end{proof}

\begin{definition}
We define a function $q:[0,1]\rightarrow[0,1]$ by
$$q(x):=\frac{1-x}{1-cx}.$$
\end{definition}

\begin{proposition}\label{prop:property-q}
The function $q$ is continuous and monotone decreasing, and $q(q(x))=x$. For every $x,y\in [0,1]$, we also have
\[
\frac{x(1-y)}{\sqrt{(1-cx)(1-cy)}}=\frac{q(y)(1-q(x))}{\sqrt{(1-cq(x))(1-cq(y)}}.
\]
\end{proposition}

\begin{proof}
The definition of $q$ reduces to $1-x-q(x)+cxq(x)=0$, from which it is clear that $q$ is an involution. Furthermore, we have the equivalent identities $(1-q(x))=x(1-cq(x))$, and $(1-cx)(1-q(x))=(1-c)x$. Multiplying the two identities gives
$$(1-c)\cdot x^{2}(1-cq(x))=(1-q(x))^{2}(1-cx),$$
and hence
\begin{equation}\label{eq:property-q-1}
\frac{x}{\sqrt{1-cx}}=\frac{1}{\sqrt{1-c}}\cdot\frac{1-q(x)}{\sqrt{1-cq(x)}}.
\end{equation}
Substituting $q(y)$ for $x$ gives
\begin{equation}\label{eq:property-q-2}
\frac{1}{\sqrt{1-c}}\cdot\frac{1-y}{\sqrt{1-cy}}=\frac{q(y)}{\sqrt{1-cq(y)}}.
\end{equation}
Multiplying the above two identities yields the conclusion.
\end{proof}

\subsection{Proof of \Cref{lem:FKG-main}}

Now we are ready to prove \Cref{lem:FKG-main}.

\begin{proof}[Proof of \Cref{lem:FKG-main}]

Using \Cref{prop:integration-h} and the Fubini-Tonelli theorem, we have
\begin{align}
K(\alpha,\alpha^{\R})&=\int_{0}^{1}\int_{0}^{1}\min\{x(1-y),y(1-x)\}\odif \alpha(x)\odif \alpha(y)
\nonumber\\
&=\int_{0}^{1}\int_{0}^{1}\sqrt{(1-cx)(1-cy)}\int_{0}^{\min\{x,y\}}\int_{\max\{x,y\}}^{1}h(r,s)\odif s\odif r\odif \alpha(x)\odif \alpha(y)\nonumber\\
&=\int_{0}^{1}\int_{0}^{s}h(r,s)\int_{r}^{s}\int_{r}^{s}\sqrt{(1-cx)(1-cy)}\odif \alpha(x)\odif \alpha(y)\odif r\odif s\nonumber\\
&=\int_{0}^{1}\int_{0}^{s}h(r,s)\left(\int_{r}^{s}\sqrt{1-cx}\odif \alpha(x)\right)^{2}\odif r\odif s.\label{eq:FKG-trans-1}
\end{align}

Note that \Cref{prop:integration-h,prop:property-q} together imply
\begin{align*}
\int_{0}^{\min\{x,y\}}\int_{\max\{x,y\}}^{1}h(r,s)\odif s\odif r&=\int_{0}^{q(\max\{x,y\})}\int_{q(\min\{x,y\})}^{1}h(r,s)\odif s\odif r\\
&=\int_{0}^{\min\{q(x),q(y)\}}\int_{\max\{q(x),q(y)\}}^{1}h(r,s)\odif s\odif r.
\end{align*}
Therefore, we can modify the calculation in \eqref{eq:FKG-trans-1} and obtain
\begin{align}
K(\beta,\beta^{\R})&=\int_{0}^{1}\int_{0}^{1}\min\{x(1-y),y(1-x)\}\odif \beta(x)\odif \beta(y)
\nonumber\\
&=\int_{0}^{1}\int_{0}^{1}\sqrt{(1-cx)(1-cy)}\int_{0}^{\min\{q(x),q(y)\}}\int_{\max\{q(x),q(y)\}}^{1}h(r,s)\odif s\odif r\odif\beta(x)\odif \beta(y)\nonumber\\
&=\int_{0}^{1}\int_{0}^{s}h(r,s)\int_{0}^{1}\int_{0}^{1}\ind{r\leq q(x),q(y)\leq s}\sqrt{(1-cx)(1-cy)}\odif \beta(x)\odif \beta(y)\odif r\odif s\nonumber\\
&=\int_{0}^{1}\int_{0}^{s}h(r,s)\int_{0}^{1}\int_{0}^{1}\ind{q(s)\leq x,y\leq q(r)}\sqrt{(1-cx)(1-cy)}\odif \beta(x)\odif \beta(y)\odif r\odif s\nonumber\\
&=\int_{0}^{1}\int_{0}^{s}h(r,s)\left(\int_{q(s)}^{q(r)}\sqrt{1-cx}\odif \beta(x)\right)^{2}\odif r\odif s.\label{eq:FKG-trans-2}
\end{align}

Now we can apply Cauchy-Schwarz to the product of \eqref{eq:FKG-trans-1} and \eqref{eq:FKG-trans-2}, followed by another Fubini-Tonelli transformation to obtain the conclusion. 
\begin{align*}
&\quad \sqrt{K(\alpha,\alpha^{\R})K(\beta,\beta^{\R})}\\
&\geq 
\int_{0}^{1}\int_{0}^{s}h(r,s)\left(\int_{r}^{s}\sqrt{1-cx}\odif \alpha(x)\right)\left(\int_{q(s)}^{q(r)}\sqrt{1-cy}\odif \beta(y)\right)\odif r\odif s\\
&=\int_{0}^{1}\int_{0}^{1} \sqrt{(1-cx)(1-cy)}\int_{0}^{1}\int_{0}^{1}\ind{r\leq x\leq s,\;q(s)\leq y\leq q(r)}h(r,s)\odif r\odif s\odif \alpha(x)\odif \beta(y)\\
&=\int_{0}^{1}\int_{0}^{1} \sqrt{(1-cx)(1-cy)}\int_{0}^{1}\int_{0}^{1}\ind{r\leq x\leq s,\;r\leq q(y)\leq s}h(r,s)\odif r\odif s\odif \alpha(x)\odif \beta(y)\\
&=\int_{0}^{1}\int_{0}^{1}\ind{x\leq q(y)}\sqrt{(1-cx)(1-cy)}\cdot\frac{x(1-q(y))}{\sqrt{(1-cx)(1-cq(y))}}\odif \alpha(x)\odif \beta(y)\\
&\quad + \int_{0}^{1}\int_{0}^{1}\ind{x> q(y)}\sqrt{(1-cx)(1-cy)}\cdot\frac{q(y)(1-x)}{\sqrt{(1-cx)(1-cq(y))}}\odif \alpha(x)\odif \beta(y)\\
&=\int_{0}^{1}\int_{0}^{1}\ind{x\leq q(y)}\cdot \sqrt{1-c}\cdot xy\odif \alpha(x)\odif \beta(y) \tag{using equation \eqref{eq:property-q-1}}\\
&\quad +\int_{0}^{1}\int_{0}^{1}\ind{x> q(y)}\cdot\frac{(1-x)(1-y)}{\sqrt{1-c}}\odif \alpha(x)\odif \beta(y) \tag{using equation \eqref{eq:property-q-2}}\\
&=\int_{0}^{1}\int_{0}^{1} \min\left\{\sqrt{1-c}\cdot xy,\frac{(1-x)(1-y)}{\sqrt{1-c}}\right\}\odif \alpha(x)\odif \beta(y)=K_{c}(\alpha,\beta).
\end{align*}
In the second to last equality above, we used the fact that
\[
\sqrt{1-c}\cdot xy\leq \frac{(1-x)(1-y)}{\sqrt{1-c}} \quad\text{if and only if}\quad x\leq q(y). \qedhere
\]
\end{proof}

\section{From directed to undirected isoperimetry}\label{sec:directed-to-undirected}

In this section, we provide a proof of \Cref{thm:FKG-to-spectral-gap} assuming \Cref{thm:directed}. Throughout the section, we fix a monotone set $A\subseteq\{0,1\}^{n}$ with at least 2 elements.

\subsection{Domain extension}

Since the target result, \Cref{thm:FKG-to-spectral-gap}, focuses solely on the subset $A$ of the hypercube, while \Cref{thm:directed} applies only to functions defined on the entire hypercube, we first introduce a simple method for extending function domains to the whole hypercube.

\begin{definition}\label{def:domain-extension}
We define an operator $T$ that extends any function $f:A\rightarrow\mathbb{R}$ to the function $T[f]:\{0,1\}^{n}\rightarrow\mathbb{R}$ defined by
$$T[f](x)=\begin{cases}
\min_{y\in A}f(y), &\text{if }x\not\in A,\\
f(x),&\text{if }x\in A.
\end{cases}$$
\end{definition}

By defining the value of the function outside of the original domain $A$ to be sufficiently small, the extension operator enjoys the following two useful properties that allow us to access the power of \Cref{thm:directed}.

\begin{proposition}\label{prop:Undirected-A-to-directed-full}
For every function $f:A\rightarrow\bR$ we have
$$\mu(A)\cdot \cE_{A}(f)=\cE^{-}(T[f])+\cE^{-}(T[-f]).$$
\end{proposition}

\begin{proof}
Since $T[f]$ is constant on $\{0,1\}^{n}\setminus A$, and since $T[f](x)\leq T[f](y)$ for all $x\in \{0,1\}^{n}\setminus A$ and $y\in A$, we know that for $x\in\{0,1\}^{n}$ and $i\in [n]$,
\[
T[f](x^{i\rightarrow 1})<T[f](x^{i\rightarrow 0})\quad\text{can hold only if}\quad x,x^{\oplus i}\in A.
\]
So we have
\begin{align*}
\cE^{-}(T[f])&=\frac{1}{4}\cdot \Exu{x\in \{0,1\}^{n}}{\sum_{i=1}^{n}\min\left\{0,T[f](x^{i\rightarrow 1})-T[f](x^{i\rightarrow 0})\right\}^{2}}\\
&=\frac{1}{4}\cdot\Exu{x\in \{0,1\}^{n}}{\sum_{i=1}^{n}\min\left\{0,T[f](x^{i\rightarrow 1})-T[f](x^{i\rightarrow 0})\right\}^{2}\cdot\ind{x,x^{\oplus i}\in A}}\\
&=\frac{\mu(A)}{4}\cdot\Exu{x\in A}{\sum_{i=1}^{n}\min\left\{0,f(x^{i\rightarrow 1})-f(x^{i\rightarrow 0})\right\}^{2}\cdot\ind{x^{\oplus i}\in A}}
\end{align*}
Applying the above argument to $T[-f]$ instead of $T[f]$, we obtain
\begin{align*}
\cE^{-}(T[-f])&=\frac{\mu(A)}{4}\cdot\Exu{x\in A}{\sum_{i=1}^{n}\min\left\{0,-f(x^{i\rightarrow 1})+f(x^{i\rightarrow 0})\right\}^{2}\cdot\ind{x^{\oplus i}\in A}}\\
&=\frac{\mu(A)}{4}\cdot\Exu{x\in A}{\sum_{i=1}^{n}\max\left\{0,f(x^{i\rightarrow 1})-f(x^{i\rightarrow 0})\right\}^{2}\cdot\ind{x^{\oplus i}\in A}}.
\end{align*}
Adding the above two equations together yields
\begin{align*}
\cE^{-}(T[f])+\cE^{-}(T[-f])&=\frac{\mu(A)}{4}\cdot\Exu{x\in A}{\sum_{i=1}^{n}\left(f(x^{i\rightarrow 1})-f(x^{i\rightarrow 0})\right)^{2}\cdot\ind{x^{\oplus i}\in A}}\\
&=\mu(A)\cdot \cE_{A}(f). \qedhere
\end{align*}
\end{proof}

\begin{proposition}\label{prop:dist-to-mono-under-extension}
For every function $f:A\rightarrow\mathbb{R}$, there exists a monotone increasing function $g:A\rightarrow\mathbb{R}$ such that 
$$\left\|f-g\right\|_{2}\leq \mu(A)^{-1/2}\cdot\dist_{2}^{\mono}(T[f]),$$ where the $L^{2}$-norm is the norm in the inner product space $L^{2}(A)$.
\end{proposition}

\begin{proof}
Since the collection of all monotone increasing real-valued functions on $\{0,1\}^{n}$ form a closed set in the Euclidean space $\mathbb{R}^{\{0,1\}^{n}}$, there exists a monotone increasing function $\widetilde{g}:\{0,1\}^{n}\rightarrow\mathbb{R}$ such that $\left\|T[f]-\widetilde{g}\right\|_{2}=\dist_{2}^{\mono}(T[f])$, where the $L^{2}$-norm is the norm in the space $L^{2}(\{0,1\}^{n})$. Now note that the restriction $g\define \widetilde{g}|_{A}$ is a monotone increasing function on $A$. Therefore,
\begin{align*}
\left\|f-g\right\|_{2}^{2}&=\Exu{x\in A}{(f(x)-g(x))^{2}}=\Exu{x\in A}{\Big(T[f](x)-\widetilde{g}(x)\Big)^{2}}\\
&\leq \mu(A)^{-1}\cdot\Exu{x\in \{0,1\}^{n}}{\Big(T[f](x)-\widetilde{g}(x)\Big)^{2}}=\mu(A)^{-1}\cdot\dist_{2}^{\mono}(T[f])^{2}.\qedhere
\end{align*}
\end{proof}

\subsection{Correlation analysis}

In this subsection, we lay some groundwork about correlation of functions (or equivalently, random variables) that will help prove \Cref{thm:FKG-to-spectral-gap}. We begin with the following natural definition of correlation ratios.

\begin{definition}\label{def:rho}
For non-constant functions $g,h:A\rightarrow\mathbb{R}$, we define
\[
\rho(g,h)\define \frac{\Covs{A}{g}{h}}{\sqrt{\Vars{A}{g}\cdot\Vars{A}{h}}}.
\]
\end{definition}

The following triangle-inequality-type lemma is going to be important in the proof of \Cref{thm:FKG-to-spectral-gap}. Conceptually, the lemma says that if functions $g$ and $h$ on $A$ are not very correlated with each other (that is, $\rho(g,h)$ is bounded away from 1), then $f$ cannot be very correlated with both $g$ and $h$ at the same time. In particular, we will later use the lemma in the case where $g$ is a monotone increasing function and $h$ is a monotone decreasing function, which cannot be very correlated if $\delta(A)$ is bounded away from $-1$.

\begin{proposition}\label{prop:trigonometry}
Consider three non-constant functions $f,g,h:A\rightarrow\bR$. We have
\[
\max\{0,\rho(f,g)\}^{2}+\max\{0,\rho(f,h)\}^{2}\leq 1+\max\{0,\rho(g,h)\}.
\]
\end{proposition}

\begin{proof}
We may without loss generality assume that $\Vars{A}{f}=\Vars{A}{g}=\Vars{A}{h}=1$. In this case, $\Covs{A}{f}{g}=\rho(f,g)$, $\Covs{A}{f}{h}=\rho(f,h)$ and $\Covs{A}{g}{h}=\rho(g,h)$.

If $\rho(f,g)<0$ then the conclusion trivially holds since $\max\{0,\rho(f,h)\}^{2}\leq 1$. Similarly if $\rho(f,h)<0$, the conclusion is also trivial. In the following, we assume that $\rho(f,g)\geq 0$ and $\rho(f,h)\geq 0$. 

Consider the matrix
\[
B\define\begin{bmatrix}
1 & \rho(f,g) & \rho(f,h)\\
\rho(f,g) & 1 & \rho(g,h)\\
\rho(f,h) & \rho(g,h) & 1
\end{bmatrix}.
\]
For each vector $\boldsymbol{\lambda}=(\lambda_{1},\lambda_{2},\lambda_{3})\in\mathbb{R}^{3}$, we know $\boldsymbol{\lambda}^{T}B\boldsymbol{\lambda}=\Vars{A}{\lambda_{1}f+\lambda_{2}g+\lambda_{3}h}\geq 0$. So $B$ is a positive semi-definite matrix. This means $\det B\geq 0$, and we can expand it into
\begin{equation}\label{eq:determinant}
1+2\rho(f,g)\rho(f,h)\rho(g,h)\geq \rho(f,g)^{2}+\rho(f,h)^{2}+\rho(g,h)^{2}.
\end{equation}
If $\rho(g,h)<0$, then \eqref{eq:determinant} implies $1\geq \rho(f,g)^{2}+\rho(f,h)^{2}$ and we arrive at the conclusion. In the following we assume $\rho(g,h)\geq 0$.

Expanding the Cauchy-Schwarz inequality $\Vars{A}{f}\cdot\Vars{A}{g+h}\geq \Covs{A}{f}{g+h}^{2}$, we have
\begin{equation}\label{eq:rank-2-determinant}
2+2\rho(g,h)\geq (\rho(f,g)+\rho(f,h))^{2}\geq 4\rho(f,g)\rho(f,h).
\end{equation}
Multiplying both sides of \eqref{eq:rank-2-determinant} by $\rho(g,h)/2$ and then adding it to \eqref{eq:determinant}, we get the desired conclusion
\[1+\rho(g,h)\geq \rho(f,g)^{2}+\rho(f,h)^{2}.\qedhere\]
\end{proof}

The following definition serves to interpret correlation ratios in terms of $L^{2}$ distances.

\begin{definition}\label{def:tau}
For functions $f,g:A\rightarrow\mathbb{R}$, we define
\[
\tau(f,g)\define \min_{a\in\bR_{\geq 0}, b\in\bR}\left\|f-(ag+b)\right\|_{2},
\]
where the $L^{2}$-norm is the norm in the inner product space $L^{2}(A)$.
\end{definition}
\begin{proposition}\label{prop:tau-to-variance}
Consider two non-constant functions $f,g:A\rightarrow\bR$. We have
\[
\tau(f,g)^{2}=\Big(1-\max\{0,\rho(f,g)\}^{2}\Big)\cdot\Vars{A}{f}.
\]
\end{proposition}
\begin{proof}
Note that
\begin{align}
\tau(f,g)^{2}&=\min_{a\in\bR_{\geq 0}, b\in\bR}\left\|f-(ag+b)\right\|_{2}^{2}=\min_{a\in\bR_{\geq 0}}\Vars{A}{f-ag}\nonumber\\
&=\min_{a\in\bR_{\geq 0}}\Big(a^{2}\cdot \Vars{A}{g}-2a\cdot \Covs{A}{f}{g}+\Vars{A}{f}\Big).\label{eq:quadratic-form-of-tau}
\end{align}

If $\rho(f,g)<0$, then $\Covs{A}{f}{g}<0$, and the quadratic polynomial in the right hand side of \eqref{eq:quadratic-form-of-tau} is minimized at $a=0$. Therefore $\tau(f,g)^{2}=\Vars{A}{f}$, as desired.

If $\rho(f,g)\geq 0$, then $\Covs{A}{f}{g}\geq 0$, and the quadratic polynomial in the right hand side of \eqref{eq:quadratic-form-of-tau} is minimized at $a=\Covs{A}{f}{g}/\Vars{A}{g}$. Therefore \eqref{eq:quadratic-form-of-tau} simplifies to
\[
\tau(f,g)^{2}=-\frac{\Covs{A}{f}{g}^{2}}{\Vars{A}{g}}+\Vars{A}{f}=(1-\rho(f,g)^{2})\cdot \Vars{A}{f},
\]
as desired.
\end{proof}

\subsection{Proof of \Cref{thm:FKG-to-spectral-gap}}

We are now ready to prove \Cref{thm:FKG-to-spectral-gap} assuming \Cref{thm:directed}.

\begin{proof}[Proof of \Cref{thm:FKG-to-spectral-gap} assuming \Cref{thm:directed}]
We have
\begin{align}
\cE_{A}(f)&=\mu(A)^{-1}\cdot\cE^{-}(T[f])+\mu(A)^{-1}\cdot\cE^{-}(T[-f])&(\text{by \Cref{prop:Undirected-A-to-directed-full}})\nonumber\\
&\geq \mu(A)^{-1}\cdot\dist_{2}^{\mono}(T[f])^{2}+\mu(A)^{-1}\cdot\dist_{2}^{\mono}(T[-f])^{2}&(\text{by \Cref{thm:directed}})\nonumber\\
&\geq \left\|f-g_{0}\right\|_{2}^{2}+\left\|-f-h_{0}\right\|_{2}^{2}&(\text{by \Cref{prop:dist-to-mono-under-extension}}),\label{eq:bound-by-g-naught}
\end{align}
for some monotone increasing functions $g_{0},h_{0}:A\rightarrow\bR$. If $g_{0}$ is non-constant, we pick $g:A\rightarrow\bR$ to be $g\define g_{0}$. If $g_{0}$ is constant, we pick an arbitrary non-constant increasing function $g:A\rightarrow\bR$. In either case, we trivially have
\[
\left\|f-g_{0}\right\|_{2}^{2}\geq \min_{a\in\bR_{\geq 0},b\in\bR}\left\|f-(ag+b)\right\|_{2}^{2}=\tau(f,g)^{2}.
\]
Similarly we pick a non-constant increasing function $h:A\rightarrow\bR$ such that $\left\|-f-h_{0}\right\|_{2}^{2}\geq \tau(-f,h)^{2}$. We can then continue from \eqref{eq:bound-by-g-naught} and have
\begin{align*}
&\quad\cE_{A}(f)\geq \tau(f,g)^{2}+\tau(-f,h)^{2}\\
&= \Big(1-\max\{0,\rho(f,g)\}^{2}\Big)\cdot\Vars{A}{f}+\Big(1-\max\{0,\rho(-f,h)\}^{2}\Big)\cdot\Vars{A}{f} &(\text{by \Cref{prop:tau-to-variance}})\\
&=\Big(2-\max\{0,\rho(f,g)\}^{2}-\max\{0,\rho(f,-h)\}^{2}\Big)\cdot\Vars{A}{f}\\
&\geq \Big(1-\max\{0,\rho(g,-h)\}\Big)\cdot\Vars{A}{f}&(\text{by \Cref{prop:trigonometry}})\\
&=\Big(1+\min\{0,\rho(g,h)\}\Big)\cdot\Vars{A}{f}\geq (1+\delta(A))\cdot\Vars{A}{f} &(\text{by \Cref{def:FKG-ratio}}).&\qedhere
\end{align*}
\end{proof}

\section{Spectral theory and heat flow for directed graphs}\label{sec:spectral-directed}

In this section, as a first step toward proving our directed Poincaré inequality for the hypercube
(\cref{thm:directed}), we first set up a framework that applies to the more general case of directed
weighted graphs. Specifically, we revisit and extend the study of directed analogues of classical
concepts from spectral graph theory such as the Laplacian operator, the Dirichlet energy, and the
heat flow; define a directed notion of spectral gap for weighed directed graphs; and show that
bounding this \emph{dynamical spectral gap} suffices for proving a directed Poincaré inequality.
Then, in the next section, \cref{thm:directed} will follow as an application once we establish a
bound on the directed spectral graph of the directed hypercube graph.

\paragraph*{Prior work on spectral theory for directed graphs.} There has been extensive prior work
developing spectral graph theory beyond the classical setting of undirected graphs, toward capturing
directed graphs and hypergraphs. Early work of \cite{Fil91,Chu05} associated a certain Hermitian
matrix with each directed graph, and showed a Cheeger-type inequality based on the eigenvalues of
that matrix, and many subsequent works have built upon that foundation; we refer to the recent
thesis \cite{Tun25} for a thorough review, and here we mention two recent lines of work that are
closest to our setting. One line of works \cite{LTW23,LTW24,Tun25} has developed a theory of
\emph{reweighted eigenvalues} capturing expansion properties of directed graphs and hypergraphs,
proved Cheeger inequalities for these settings, and devised efficient algorithms for graph
partitioning. Another line of works \cite{Yos16,Yos19,FSY21,IMTY22} has pursued similar goals by
analyzing a \emph{nonlinear Laplacian} operator and the heat equation associated with it.

While our interest in a spectral theory for directed graphs is related to these previous works (and
indeed we will build upon the approach of \cite{Yos16}), our focus is slightly different. In a
nutshell, while prior works have focused on spectral characterizations of good expansion of a
directed graph $G$ as captured by directed versions of Cheeger inequalities (for edge conductance or
vertex expansion) and mixing time of random walks, our focus will be on the quality of $G$ as the
``substrate'' for a dynamical process; we will consider $G$ a good directed spectral expander if it
affords fast convergence for that process. In particular, our perspective allows for \emph{directed
acyclic graphs} to be considered good expanders, which is a stark departure from prior perspectives
-- as we briefly explain next.

Indeed, a central focus of prior works has been to establish Cheeger inequalities of the type
\begin{equation}
    \label{eq:directed-cheeger}
    \vec{\lambda}_2 \lesssim \vec{\phi}(G) \lesssim \sqrt{\vec{\lambda}_2} \,,
\end{equation}
where $\vec{\lambda}_2$ denotes a relevant second eigenvalue related to the directed weighted graph
$G$, and the edge conductance $\vec{\phi}(G)$ of $G$ is
\[
    \vec{\phi}(G) \define \min_{\emptyset \ne S \subsetneq V}
    \frac{\min\left\{ w(\delta^+(S)), w(\delta^+(V \setminus S)) \right\}}{
    \min\left\{ \vol_w(S), \vol_w(V \setminus S) \right\}} \,,
\]
where $\delta^+(S)$ denotes the outgoing edge boundary of $S$ and $\vol_w(S)$ denotes the total
weighted degree of all vertices in $S$. Now, if $G$ is not strongly connected, then in general there
exists a set $S$ with positive volume but no outgoing edges, which makes $\vec{\phi}(G)$ and thus
$\vec{\lambda}_2$ zero. In particular, this is the case for the directed hypercube graph which we
are interested in, so if we hope to show a non-trivial directed Poincaré inequality via a spectral
argument, such a quantity $\vec{\lambda}_2$ will not do.

\paragraph*{Our approach.}
The type of spectral theory for directed graphs we study in this section was first developed by
\cite{Yos16} in the context of network analysis. In that work, \cite{Yos16} defined a nonlinear
Laplacian operator acting on real-valued functions defined on the vertices of a directed graph,
showed that this operator induces a dynamical process that is a directed analogue of the classical
heat flow on graphs, proved that this operator has nontrivial eigenvalues, and established a Cheeger
inequality like \eqref{eq:directed-cheeger} for this setting. While \cite{Yos16} focused on the
implications of directed spectral theory for graph partitioning and related problems in network
analysis, we focus on the dynamical properties of the heat flow on directed graphs -- namely its
convergence to a \emph{monotone} limit, and its connection to the directed Poincaré inequality.

Let us briefly motivate and preview the main ideas in our argument. In classical spectral graph
theory, given a graph $G = (V, E)$, the following four concepts play a central role:
\begin{enumerate}
    \item The Laplacian operator $\lap$, which acts on a function $f$ by outputting another function
        $\lap f$.
    \item The Dirichlet energy functional $\dir$, which associates with each $f$ an energy $\cE(f)
        \ge 0$ measuring the ``local variance'' of $f$ along edges of $G$.
    \item The heat flow semigroup $S_t$, which captures a dynamical process which starts at some
        initial state $f$ and has its rate of change governed by the Laplacian: $\odv*{S_t f}{t} =
        \lap S_t f$. The heat flow informally ``sends mass'' along each edge of $G$ from the vertex
        with higher $f$-value to the vertex with lower $f$-value, causing the system to converge to
        an equilibrium state.
    \item The Poincaré inequality, which states that $\Var{f} \le \frac{1}{\lambda} \, \cE(f)$. The
        best constant $\lambda$ is called the \emph{spectral gap} of $G$.
\end{enumerate}

The appearance of the Poincaré inequality above hints at the relevance of this theory to our goal of
proving a directed Poincaré inequality, and we mentioned above that our strategy toward this goal
will be to define a directed version of the spectral gap. As a motivation for this strategy, we
recall that the spectral gap ties together essentially all of the elements of the list above;
indeed, as summarized in \cite[Theorem~2.18]{Han14}, the following are equivalent given a constant
$c \ge 0$:
\begin{enumerate}
    \item \label{item:poincare-inequality}
        Poincaré inequality: $\Var{f} \le c \dir(f)$ for all $f$.
    \item \label{item:variance-decay}
        Variance decay: $\Var{S_t f} \le e^{-2t/c} \Var{f}$ for all $f, t$.
    \item \label{item:energy-decay}
        Energy decay: $\dir(S_t f) \le e^{-2t/c} \dir(f)$ for all $f, t$.
\end{enumerate}

To prove a directed Poincaré inequality, we replace the Laplacian operator $\lap$ with an operator
$\lap^-$ which, intuitively, only ``sends mass'' from vertex $u$ to vertex $v$ if $(u,v)$ is a
directed edge and $f(u) > f(v)$, \ie $f$ violates monotonicity along edge $(u,v)$; the dynamical
system induced by $\lap^-$ is the \emph{directed heat flow} on $G$, which was studied by
\cite{Yos16}. The directed heat flow is precisely the \emph{gradient system} for the directed energy
functional $\dir^-(f)$, \ie the upward boundary from \cref{def:upward-boundary}. Thus, this system
intuitively ``corrects'' the local violations of monotonicity as quickly as possible, and indeed it
converges to a monotone function as $t \to \infty$.

This directed theory cannot fully analogize the classical situation above; for example, variance
decay fails to hold, because non-constant monotone functions are (non-unique!) stationary solutions.
Instead, we will \emph{define} the \emph{dynamical spectral gap} of $G$ as the best constant
characterizing the (directed) energy decay, as in \cref{item:energy-decay} above, and then show that
(the directed version of) \cref{item:energy-decay} implies (the directed version of)
\cref{item:poincare-inequality}.

As mentioned in the introduction, recent work of \cite{pinto2024directed} also proved a directed
Poincaré inequality -- for functions defined on the continuous cube $[0,1]^n$ -- using a dynamical
argument. Indeed, that work also took as its starting point the connections between the heat flow
and the Poincaré inequality, and showed that the natural directed version of the heat flow in
continuous space enjoys exponential energy decay, which implies a directed Poincaré inequality for
that setting. Our proof is conceptually similar to the proof of \cite{pinto2024directed}, but our techniques differ
in at least two ways: 1)~\cite{pinto2024directed} required analytical arguments
from the theory of partial differential equations (PDEs), while we are able to study our dynamical
process as an ordinary differential equation (ODE) thanks to the finite-dimensional nature of our
problem; and 2)~\cite{pinto2024directed} used tools from optimal transport theory to tensorize their
one-dimensional result, while we obtain a multidimensional inequality directly
by studying the directed heat flow as a gradient system. In this last
regard, our proof also bears resemblance to, and is inspired by prior work of \cite{KLLR18} on the
so-called \emph{Paulsen problem} from operator theory, where a ``movement decay'' property of a
suitable dynamical system was used to bound the distance between the initial and equilibrium states
of that system.

\paragraph*{Organization.}
The rest of this section is organized as follows.
\cref{subsection:laplacian,subsec:energy-functional,subsection:heat-flow} present the directed
versions of the Laplacian operator, the Dirichlet energy functional, and the heat flow,
respectively. These subsections are largely an alternative exposition of the ideas covered in
\cite{Yos16}, but with a different emphasis tailored to our goals\footnote{In particular,
\cite{Yos16} defined both normalized and unnormalized versions of their Laplacian operator, and
focused on the normalized one. We study a single definition for weighted graphs.}\!\!. Then, in
\cref{subsection:spectral-gap} we define the dynamical spectral gap of a directed weighted graph,
and show via an energy decay argument that every directed weighted graph admits a directed Poincaré
inequality mediated by the dynamical spectral gap.

\paragraph{Notation.} Given a discrete set and an element $u \in V$, we write $e_u \in L^2(V)$ for the standard basis vector given by $e_u(v) \define
\ind{u=v}$.

\subsection{The directed Laplacian}
\label{subsection:laplacian}

Let $G = (V, w)$ be a directed weighted graph, where $w : V \times V \to [0, +\infty)$ is function
specifying the weights of edges in $G$. By convention, we say that $(u, v) \in V \times V$ is an
edge in $G$ when $w(u, v) > 0$. We say $G$ is \emph{undirected} is $w$ is a symmetric function.

\begin{definition}[Directed Laplacian \cite{Yos16}]
    \label{def:laplacian-general}
    The \emph{directed Laplacian operator} of $G$ is the operator $\lap^- = \lap^-_{G} : L^2(V) \to
    L^2(V)$ given by
    \begin{equation}
        \lap^-f \define \frac{1}{2} \sum_{u,v \in V} w(u,v) \left(f(u) - f(v)\right)^+ (e_v - e_u)
    \end{equation}
    for each $f \in L^2(V)$. In this paper, we use the notation $x^{+}:=\max\{x,0\}$, for $x\in \mathbb{R}$.
\end{definition}

Given $f \in L^2(V)$, we say an edge $(u, v)$ is \emph{$f$-monotone} if $f(u) \le f(v)$, and we say
it is \emph{$f$-antimonotone} if $f(u) > f(v)$. We say \emph{$f$ is monotone} if every edge $(u,v)$
of $G$ is $f$-monotone. If we think of $f$ as the distribution of ``mass'' over the vertices $V$ and of
$\lap^- \bm{f}$ as the rate of change of $\bm{f} = \bm{f}(t)$ over time $t$, then
\cref{def:laplacian-general} posits that mass flows along the $f$-antimonotone edges, from the
heavier vertex to the lighter one. When $G$ is undirected, this process is the standard heat flow on
$G$, and indeed \cref{def:laplacian-general} recovers the standard graph Laplacian in this case:

\begin{observation}
    \label{obs:standard-laplacian}
    If $G$ is undirected, then $\lap^-_{G}$ is (half of) the standard (unnormalized)
    Laplacian operator $\lap_{G}$ of $G$. Indeed, we can see the action of $\lap^-_{G}$ on $f \in L^2(V)$
    as follows: 1)~remove the $f$-monotone edges $(u,v)$ from $G$; 2)~view the resulting graph $G'$ as undirected; and 3)~apply the standard Laplacian operator $\lap_{G'}$ to $f$.
\end{observation}

\begin{remark}
    In spectral graph theory, one typically defines
    the Laplacian operator of an undirected graph as $-\lap_{G}$ in our notation, \ie by replacing
    $e_v - e_u$ with $e_u - e_v$ in \cref{def:laplacian-general}. Our notation follows instead the
    tradition from probability theory (see \eg \cite{BGL14,Han14}), which has the advantages 1)~that
    $\lap^-$ itself, rather than $-\lap^-$, will be the generator of the heat semigroup -- our main
    object of interest; and 2)~of consistency with the analytic setting, where we have the Laplacian
    operator $\Delta$ for smooth functions in Euclidean space.
\end{remark}

We note that unlike the standard Laplacian operator, the operator $\lap^-$ is nonlinear and not self-adjoint in general. Instead, one may think of $\lap^{-}$ as a ``piecewise linear'' operator on $L^{2}(V)$, which is in particular Lipschitz continuous.

\begin{lemma}
    \label{lemma:laplacian-lipschitz}
    The operator $\lap^- : L^2(V) \to L^2(V)$ is Lipschitz continuous.
\end{lemma}
\begin{proof}
    Since all norms are equivalent in a finite-dimensional space, it suffices to let $f, g \in
    L^2(V)$ differ on a single point $z \in V$ and show that, for all $u \in V$, $\abs*{(\lap^-
    f)(u) - (\lap^- g)(u)} \le M\abs*{f(z) - g(z)}$ for some constant $M > 0$, which may depend on
    the graph $G$ but not on $f$ or $g$.

    By \cref{def:laplacian-general}, we have
    \[
        (\lap^- f)(u)
        = \frac{1}{2} \sum_{v \in V} \left[
            -w(u, v) \left(f(u)-f(v)\right)^+ + w(v, u) \left(f(v)-f(u)\right)^+
        \right]
    \]
    and
    \[
        (\lap^- g)(u)
        = \frac{1}{2} \sum_{v \in V} \left[
            -w(u, v) \left(g(u)-g(v)\right)^+ + w(v, u) \left(g(v)-g(u)\right)^+
        \right] \,.
    \]
    By the triangle inequality and the assumption that $f$ and $g$ differ only on coordinate $z$,
    \begin{align*}
        \abs*{(\lap^- f)(u) - (\lap^- g)(u)}
        &\le \frac{1}{2} \sum_{v,v' \in V : z \in \{v,v'\}} w(v,v')
            \abs*{\left(f(v)-f(v')\right)^+ - \left(g(v)-g(v')\right)^+} \\
        &\le \frac{1}{2} \sum_{v,v' \in V : z \in \{v,v'\}} w(v,v')
            \abs*{\left(f(v)-f(v')\right) - \left(g(v)-g(v')\right)} \\
        &\le \frac{1}{2} \abs*{f(z)-g(z)} \sum_{v,v' \in V} w(v,v') \,. \qedhere
    \end{align*}
\end{proof}

\subsection{The energy functional}\label{subsec:energy-functional}

As in the case of the standard Laplacian operator, the directed Laplacian naturally induces an
\emph{energy functional} (or Dirichlet form). The following definition corresponds to the Rayleigh
quotient defined in \cite{Yos16}.

\begin{definition}[Energy functional]
    The \emph{directed Dirichlet energy} functional $\dir^- : L^2(V) \to \bR$ is given by
    \[
        \dir^-(f) \define -\inp{f}{\lap^- f} \,.
    \]
\end{definition}

The directed Dirichlet energy measures the local violations of monotonicity along edges of $G$, and
it is indeed always non-negative, as shown in the following proposition (which is similar to
Lemma~4.3 of \cite{Yos16} for the normalized nonlinear Laplacian).

\begin{proposition}[Energy functional measures local violations]
    \label{prop:energy-violations}
    For each $f \in L^2(V)$, it holds that
    \[
        \dir^-(f) = \frac{1}{2} \Exu{u \in V}{
            \sum_{v \in V} w(u,v) \left(\left(f(u)-f(v)\right)^+\right)^2
        } \,.
    \]
\end{proposition}
\begin{proof}
    We have
    \begin{align*}
        \dir^-(f)
        &= -\inp{f}{\lap^- f}
        = -\inp{f}{
                \frac{1}{2} \sum_{u',v \in V} w(u',v) \left(f(u')-f(v)\right)^+ (e_v - e_{u'})
        } \\
        &= -\frac{1}{2} \Exu{u \in V}{ f(u) \left(
                \sum_{u' \in V} w(u',u) \left(f(u')-f(u)\right)^+
                - \sum_{v \in V} w(u,v) \left(f(u)-f(v)\right)^+
        \right)} \\
        &= -\frac{1}{2} \Exu{u \in V}{ f(u) \sum_{v \in V} \Big(
            w(v,u) \left(f(v)-f(u)\right)^+ - w(u,v) \left(f(u)-f(v)\right)^+ \Big)} \\
        &= \frac{1}{2|V|} \sum_{u,v \in V} w(u,v) f(u) \left(f(u)-f(v)\right)^+
            - \frac{1}{2|V|} \sum_{u,v \in V} w(v,u) f(u) \left(f(v)-f(u)\right)^+ \\
        &= \frac{1}{2} \Exu{u \in V}{\sum_{v \in V} w(u,v) \left(\left(f(u)-f(v)\right)^+\right)^2}
        \,. \qedhere
    \end{align*}
\end{proof}

Observe that since $\lap^{-}$ is ``piecewise linear'', the energy functional $\dir^{-}$ is a ``piecewise quadratic'' functional on $L^{2}(V)$. Furthermore, at each point $f\in L^{2}(V)$ the Laplacian $\lap^{-}f$ points at the direction opposite to the gradient of $\dir^{-}$. To formalize this viewpoint, we first need to introduce some standard definitions to clarify what we mean by the gradient of a functional on 
$L^2(V)$, such as $\dir^-$. We recall the following standard definition.

\begin{definition}[Fréchet derivative and gradient]
    We say that $F : L^2(V) \to \bR$ is \emph{Fréchet differentiable} at $f \in L^2(V)$ if there
    exists a bounded linear operator $A : L^2(V) \to \bR$ such that
    \[
        \lim_{\|h\|_{2} \to 0} \frac{\abs*{F(f+h) - F(f) - Ah}}{\|h\|_{2}} = 0 \,.
    \]
    In this case, we call $DF(f) \define A$ the \emph{Fréchet derivative} of $F$ at $f$. Moreover,
    by the Riesz representation theorem there exists a unique vector $v \in L^2(V)$ satisfying
    \begin{equation}
        \label{eq:v-grad}
        Ah = \inp{h}{v}
    \end{equation}
    for all $h \in L^2(V)$. We call $\grad F(f) \define v$ the \emph{gradient} of $F$ at $f$.
    Finally, we say that $F$ is (Fréchet) $C^1$ if it is Fréchet differentiable at every $f \in
    L^2(V)$ and the map $f \mapsto DF(f)$ (from $L^2(V)$ to the space of bounded linear operators
    $L^2(V) \to \bR$) is continuous.
\end{definition}
The following proposition follows from standard real analysis.
\begin{proposition}
    \label{lemma:frechet}
    Suppose $F : L^2(V) \to \bR$ has continuous partial derivatives with respect to the standard
    basis $\{e_u\}_{u \in V}$, that is, for each $u \in V$ the $L^2(V) \to \bR$ function 
    \[
        \partial_u F(f) = \lim_{t \to 0} \frac{F(f + t e_u) - F(f)}{t}
    \]
    is continuous. Then $F$ is (Fréchet) $C^1$ and its gradient $\grad F : L^2(V) \to L^2(V)$ is
    given by
    \[
        (\grad F(f))(u) = |V| \cdot \partial_u F(f)
    \]
    for each $f \in L^2(V)$ and $u \in V$.
\end{proposition}

We highlight that since the inner product on $L^{2}(V)$ used in \eqref{eq:v-grad} differs from the standard inner product on the finite dimensional Euclidean space $\bR^{V}$ by a factor of $|V|$, the gradient of $F$ also differs from the partial derivatives of $F$ by a factor of $|V|$. 

We are now equipped to show that $\dir^-$ is a convex $C^1$ functional, and that the Laplacian
$\lap^{-}$ is (half of) the negative gradient of $\dir^{-}$.

\begin{lemma}
    \label{lemma:energy-c1}
    The functional $\dir^- : L^2(V) \to [0, +\infty)$ is $C^1$ and convex.
\end{lemma}
\begin{proof}
We define a function $g
: \bR^2 \to \bR$ by
\begin{equation}\label{eq:def-of-g}
    g(x, y) \define \frac{1}{2} \left( (x-y)^+ \right)^2 \,.
\end{equation}

    \textbf{Class $C^1$.} By \cref{prop:energy-violations}, it suffices to show that for each $u, v
    \in V$, the $L^2(V) \to \bR$ function $f \mapsto \left(\left(f(u)-f(v)\right)^+\right)^2$ is
    $C^1$. By \cref{lemma:frechet}, it suffices to show that this function has continuous partial
    derivatives, which is equivalent to showing that $g$ has continuous partial derivatives. This is
    indeed the case, with $\partial_1 g(x,y) = (x-y)^+$ and $\partial_2 g(x,y) = -(x-y)^+$.

    \textbf{Convexity.} By \cref{prop:energy-violations} and recalling that the weights $w(u, v)$
    are non-negative, it again suffices to show that the function $g$ is convex. Indeed, let $(x_1,
    y_1), (x_2, y_2) \in \bR^2$ and $\lambda \in [0,1]$. Then, using the convexity of the functions
    $z \mapsto z^+$ and $z \mapsto z^2$ along with the fact that the latter function is increasing
    for $z \ge 0$, we have
    \begin{align*}
        &g\left(\lambda (x_1,y_1) + (1-\lambda) (x_2,y_2)\right)
        = \frac{1}{2} \left[\left[(\lambda x_1 + (1-\lambda) x_2)
                - (\lambda y_1 + (1-\lambda) y_2)\right]^+\right]^2 \\
        &\quad = \frac{1}{2} \left[
            \left[ \lambda(x_1-y_1) + (1-\lambda)(x_2-y_2) \right]^+ \right]^2
        \le \frac{1}{2} \left[ \lambda(x_1-y_1)^+ + (1-\lambda)(x_2-y_2)^+ \right]^2 \\
        &\quad \le \lambda g(x_1, y_1) + (1-\lambda) g(x_2, y_2) \,.
        \qedhere
    \end{align*}
\end{proof}

\begin{lemma}\label{lem:gradient}
    For any $f\in L^{2}(V)$, we have $\lap^{-}f = -\frac{1}{2} \grad \dir^{-}(f)$.
\end{lemma}
\begin{proof}
    Recall that the function $g$ in \eqref{eq:def-of-g} has partial derivatives $\partial_1 g(x, y) = (x-y)^+$ and
    $\partial_2 g(x, y) = -(x-y)^+$. Now, by \cref{prop:energy-violations}, for each $u \in V$, the
    partial derivative $\partial_u \dir^-(f)$ (in the notation of \cref{lemma:frechet}) is
    \begin{align*}
        \partial_u \dir^-(f)
        &= \frac{1}{|V|} \sum_{v \in V} \left[
            w(u,v) \partial_1 g(f(u), f(v)) + w(v,u) \partial_2 g(f(v), f(u))
        \right] \\
        &= \frac{1}{|V|} \sum_{v \in V} \left[
            w(u,v) \left(f(u)-f(v)\right)^+ - w(v,u) \left(f(v)-f(u)\right)^+
        \right] \,.
    \end{align*}
    Thus, by \cref{lemma:frechet}, the gradient $\grad \dir^-(f)$ (which exists by
    \cref{lemma:energy-c1}) is given by
    \[
        (\grad \dir^-(f))(u) = |V| \cdot \partial_u \dir^-(f)
        = \sum_{v \in V} \left[
            w(u,v) \left(f(u)-f(v)\right)^+ - w(v,u) \left(f(v)-f(u)\right)^+
        \right]
    \]
    for each $u \in V$.
    On the other hand, by \cref{def:laplacian-general},
    \[
        (\lap^- f)(u)
        = \frac{1}{2} \sum_{v \in V} \left[
            -w(u, v) \left(f(u)-f(v)\right)^+ + w(v, u) \left(f(v)-f(u)\right)^+
        \right]
    \]
    for each $u \in V$, so $(\lap^- f)(u) = -\frac{1}{2} (\grad \dir^-(f))(u)$ as desired.
\end{proof}

\subsection{Directed heat flow}
\label{subsection:heat-flow}

As previewed in the \Cref{subsection:laplacian}, the directed Laplacian operator can be thought of as the
rate of mass transfer along $f$-antimonotone edges of $G$ in a dynamical process. Let us make this
notion precise.

Given any $f \in L^2(V)$, we define the \emph{directed heat flow} on $G$ with \emph{initial state}
$f$ as the dynamical system given by the initial value problem (IVP)

\begin{equation}
    \label{eq:ivp}
    \bm{f'}(t) = \lap^- \bm{f}(t) \quad \text{for all } t \ge 0 \,, \qquad \bm{f}(0) = f \,.
\end{equation}
Since the operator $\lap^-$ is Lipschitz by \cref{lemma:laplacian-lipschitz}, a standard existence
and uniqueness theorem for ordinary differential equations (ODEs) implies that this IVP has a unique
solution $\bm{f} : [0, +\infty) \to L^2(V)$; see \eg \cite[Corollary~2.6]{Tes12}. This can also be
shown using the theory of maximal monotone operators, as done by \cite{IMTY22} for the heat flow on
hypergraphs, and by \cite{Yos19} in a study that generalizes both the directed graph setting of
\cite{Yos16} and the hypergraph setting of \cite{IMTY22}.

Moreover, the directed heat flow enjoys the following semigroup structure. Define the operator
family $\left(P_t\right)_{t \ge 0}$, with $P_t : L^2(V) \to L^2(V)$ for each $t \ge 0$, as follows:
for each $f \in L^2(V)$, let $\bm{f} : [0, +\infty) \to L^2(V)$ be the solution to the IVP
\eqref{eq:ivp}, and let
\[
    P_t f \define \bm{f}(t) \,.
\]
Then it immediately follows that $P_t$ satisfies the properties of a semigroup, namely
\begin{enumerate}
    \item $P_0$ is the identity operator.
    \item $P_s P_t = P_{s+t}$ for all $s, t \ge 0$.
    \item $\lim_{t \to 0} P_t f = f$ for all $f \in L^2(V)$ (this follows from the differentiability
        of the solution $\bm{f}(t)$).
\end{enumerate}
We call $P_t$ the \emph{directed heat semigroup} operator. Note that $P_t$ is a nonlinear operator.

\begin{observation}[Monotone functions are stationary solutions]
    \label{obs:stationary}
    If $f \in L^2(V)$ is monotone, then $\lap^- f = 0$ and hence $P_t f = f$ for all $t \ge
    0$.
\end{observation}

Since monotone functions are stationary solutions to the directed heat flow (\cref{obs:stationary}),
while any non-monotone $f$ has non-zero Laplacian $\lap^{-}f$ (see \cref{prop:energy-violations}), it
is natural to expect that $P_t f$ always converges to a monotone function as $t \to \infty$. This is
indeed the case, because the directed heat flow is a \emph{gradient
system} for the \emph{convex} energy functional $\dir^-$.
\begin{proposition}[Directed heat flow is gradient system]
    \label{prop:gradient-system}
    For all $f \in L^2(V)$ and $t \ge 0$, we have
    \[
        \odv*{P_t f}{t} = -\frac{1}{2} \grad \dir^-(P_t f)\,.
    \]
\end{proposition}
\begin{proof}
By the definition of $P_{t}$ we have $\odv*{P_{t} f}{t}=\lap^{-}P_{t}f$. The claim then follows from \Cref{lem:gradient}.
\end{proof}
\begin{corollary}[Convergence to monotone equilibrium]
    \label{cor:convergence}
    For every $f \in L^2(V)$, there exists a (unique) monotone $f^* \in L^2(V)$ such that $P_t f \to
    f^*$ as $t \to \infty$.
\end{corollary}
\begin{proof}
    It is a standard fact that since $\dir^-$ is a convex differentiable function which attains its
    minimum, its gradient system (which is given by $P_t$) converges to a minimizer $f^*$; see \eg
    \cite[Chapter~3, Theorem~2]{AC84}. By \cref{prop:energy-violations}, such a minimizer is
    monotone.
\end{proof}

In light of \cref{cor:convergence}, we may define the following limit operator.
\begin{definition}[Monotone equilibrium]
    We define the operator $P_\infty : L^2(V) \to L^2(V)$ by $P_\infty f \define \lim_{t \to \infty}
    P_t f$ for each $f \in L^2(V)$, and call $P_\infty f$ the \emph{monotone equilibrium} of $f$.
\end{definition}

\begin{remark}
    Since $P_\infty f$ is monotone, it is constant in each strongly connected component of $G$.
\end{remark}

\subsection{Dynamical spectral gap}
\label{subsection:spectral-gap}

In this subsection, we associate with the directed Laplacian operator $\lap^- = \lap_{G}^{-}$ a
quantity $\lambda^- = \lambda^-(G)$, the \emph{dynamical spectral gap} of $G$, as a natural directed
generalization of the classical (undirected) case. In particular, $\lambda^-(G)$ characterizes the
rate of energy decay in the directed heat flow as $P_t f$ converges to its monotone equilibrium, and
this also implies a directed Poincaré inequality linking the distance between the initial and
equilibrium states to the energy of the initial state (\ie its violations of monotonicity).

In the classical theory, the spectral gap $\lambda(G)$ associated with the classical Laplacian
operator $\lap_{G}$ is
\begin{equation}
    \label{eq:classical-spectral-gap}
    \lambda(G) = \inf_{f \ne \vec{0} : \;\Ex{f}=0} \frac{-\inp{f}{\lap_{G} f}}{\|f\|_{2}^{2}} \,,
\end{equation}
which is the best constant for which the Poincaré inequality
\[
    \Var{f} \le \frac{1}{\lambda(G)} \cE(f)
\]
holds, where the classical Dirichlet energy is $\cE(f) = -\inp{f}{\lap_{G}f}$. As mentioned at the
beginning of this section, the spectral gap
also characterizes the rate of
exponential decay of the variance toward its minimum value under the heat flow. That is, $\lambda(G)$ is the best constant for
which
\begin{equation}
    \label{eq:classical-global-energy-decay}
    \odv*{\Var{\widetilde{P_t} f}}{t} \le -2\lambda(G) \Var{\widetilde{P_t} f}
\end{equation}
holds for all $f \in L^2(V)$ and $t \ge 0$, where $\widetilde{P_{t}}$ is the undirected heat flow. It also similarly characterizes the exponential decay of
the Dirichlet energy, \ie $\lambda(G)$ is the best constant for which
\begin{equation}
    \label{eq:classical-local-energy-decay}
    \odv*{\cE(\widetilde{P_t} f)}{t} \le -2\lambda(G) \cE(\widetilde{P_t} f)
\end{equation}
holds for all $f \in L^2(V)$ and $t \ge 0$. We refer the reader to \cite{Han14} for a comprehensive
reference.

We wish to define a natural notion of spectral gap for the directed Laplacian $\lap^-$, towards the
goal of obtaining a directed Poincaré inequality. One attempt is to take the directed analogue of
the ratio in \eqref{eq:classical-spectral-gap}, namely
\begin{equation}
    \label{eq:directed-attempt}
    \frac{\dir^-(f)}{\|f\|_{2}^{2}}
    = \frac{-\inp{f}{\lap^- f}}{\|f\|_{2}^{2}} \,.
\end{equation}
(This would essentially be the unnormalized version of the spectral gap defined by \cite{Yos16}
for the normalized nonlinear Laplacian.) However, this quantity faces the issue that any
monotone function $f$ has $\dir^-(f) = 0$, which makes the ratio zero\footnote{While for strongly
    connected directed graphs only constant functions are monotone, which renders the objection
void, this is not the case for general directed graphs.}\!\!. Even if we restricted our
attention to non-monotone $f$, in general $\|f\|_{2}^2$ could still be made arbitrarily large while
keeping $\dir^-(f) > 0$ bounded, by putting large values of $f(u)$ that do not introduce violations
of monotonicity. Therefore \eqref{eq:directed-attempt} does not seem like an informative quantity
for our purposes.

Another way to see the conceptual issue is that, unlike in the classical case, in the directed
setting we cannot hope to have variance decay as in \eqref{eq:classical-global-energy-decay}, since
the monotone equilibrium $P_\infty f$ may have strictly positive variance. However, as our next
attempt will reveal, we \emph{can} establish the decay of Dirichlet energy as in
\eqref{eq:classical-local-energy-decay}, which will suffice to give a directed Poincaré inequality
for $G$.

We start by observing that, in the classical case, the ratio
\begin{equation}
    \label{eq:undirected-better}
    \inf_{f : \cE(f) > 0} \frac{\|\lap f\|_{2}^2}{\cE(f)}
\end{equation}
equals the smallest \emph{non-zero} eigenvalue of the classical Laplacian $-\lap_{G}$, so for a
\emph{connected} graph it also equals the spectral gap $\lambda(G)$. Now, its directed analogue
\begin{equation}
    \label{eq:directed-better}
    \inf_{f : \dir^-(f) > 0} \frac{\|\lap^- f\|_{2}^2}{\dir^-(f)}
\end{equation}
seems like a reasonable candidate -- intuitively, note that not only does the constraint $\dir^-(f)
> 0$ rule out monotone functions $f$, but also, $f$-monotone edges $(u,v)$ with large $f(v)-f(u)$
contribute nothing to the denominator, suggesting that \eqref{eq:directed-better} does not have the
same shortcoming as \eqref{eq:directed-attempt}\footnote{
    For example, let $G$ consist of two disjoint edges $(u, v)$ and $(u', v')$, and consider the
    sequence $(f_i)_{i \in \bN}$ with $f_i(u) = 0$, $f_i(v) = i$, $f(u') = 1$, and $f(v') = 0$. Then
    $\dir^-(f_i)/\|f_i\|_{2}^2 \to 0$ as $i \to \infty$, while $\|\lap^-
    f_i\|_{2}^2/\dir^-(f_i)$ remains constant.
}\!\!. Another way to see the intuition behind \eqref{eq:directed-better} is that, while its
classical counterpart \eqref{eq:undirected-better} only differs from
\eqref{eq:classical-spectral-gap} by ignoring the zero eigenvalues coming from the single
pathological case of disconnected graphs, the directed \eqref{eq:directed-better} ignores many more
uninformative situations and captures the connectivity of $G$ in the appropriate, directed sense.

Therefore, we define

\begin{definition}
    \label{def:directed-spectral-gap}
    The \emph{dynamical spectral gap} of $G$ is the quantity $\lambda^-(G) \in [0, +\infty]$ given by
    \[
        \lambda^-(G) \define
        \inf \left\{
            \frac{\|\lap^- f\|_{2}^2}{\dir^-(f)} \;\; \Big| \;\;
                f \in L^2(V) \,,\,\, \dir^-(f) > 0
        \right\} \,.
    \]
\end{definition}

The next proposition shows that every graph has a non-zero dynamical spectral gap.

\begin{proposition}
    The dynamical spectral gap $\lambda^-(G)$ is strictly positive, and it is a real number unless
    $G$ consists of isolated vertices only.
\end{proposition}
\begin{proof}
    If $G$ contains at least one edge $e = (u, v)$, then any function $f \in L^2(V)$ with $f(u) = 1$
    and $f(v) = 0$ has $\dir^-(f) > 0$, so $\lambda^-(G) < +\infty$.

    We now show that $\lambda^-(G) > 0$. The idea is that, for each function $f \in L^2(V)$, the
    action of $\lap^-$ on $f$ is the same as the action of the standard Laplacian on a subgraph of
    $G$ (made undirected) that preserves only the $f$-antimonotone edges (as noted in
    \cref{obs:standard-laplacian}). The number of such subgraphs is at most $|V|!$, \ie the number
    of permutations induced by $f$, so $\lambda^-(G)$ is lower bounded by the minimum non-zero
    eigenvalue out of finitely many graphs.

    Formally, let $n \define |V|$, fix any $f \in L^2(V)$ satisfying $\dir^-(f) > 0$, and let $\pi :
    V \to [n]$ be a bijection such that $f({\pi^{-1}(1)}) \le f({\pi^{-1}(2)}) \le \dotsm \le
    f({\pi^{-1}(n)})$. Let $G_\pi = (V, w_\pi)$ be the undirected, weighted graph defined as follows:
    for each $u, v \in V$,
    \[
        w_\pi(u, v) \define \begin{cases}
            \frac{1}{2} w(u, v) & \text{if } \pi(u) > \pi(v) \\
            \frac{1}{2} w(v, u) & \text{otherwise.}
        \end{cases}
    \]
    Let $\lap_{\pi}\define \lap_{G_\pi}$ denote the standard Laplacian operator for $G_\pi$. We claim that
    $\lap^- f = \lap_\pi f$. First, by definition of the standard Laplacian, for any $z \in V$ we have
    \[
        (\lap_\pi f)(z) = \sum_{v \in V} w_\pi(z, v) \left(f(v) - f(z)\right) \,.
    \]
    On the other hand, \cref{def:laplacian-general} and the definition of $w_\pi$ yield
    \begin{align*}
        (\lap^- f)(z)
        &= \frac{1}{2} \sum_{u,v \in V} w(u,v) \left(f(u)-f(v)\right)^+ (e_v(z)-e_u(z)) \\
        &= \frac{1}{2} \sum_{u : \pi(u)>\pi(z)} w(u,z) \left(f(u)-f(z)\right)
            - \frac{1}{2} \sum_{v : \pi(z)>\pi(v)} w(z,v) \left(f(z)-f(v)\right) \\
        &= \sum_{v : \pi(v)>\pi(z)} w_\pi(v,z) \left(f(v)-f(z)\right)
            + \sum_{v : \pi(z)>\pi(v)} w_\pi(v,z) \left(f(v)-f(z)\right) \\
        &= \sum_{v \in V} w_\pi(v,z) \left(f(v)-f(z)\right)
        = (\lap_\pi f)(z) \,,
    \end{align*}
    so indeed $\lap^- f = \lap_\pi f$. Let $\lambda_\pi$ denote the smallest non-zero eigenvalue of
    $-\lap_\pi$, which exists because $G_\pi$ contains at least one edge since $f$ is not monotone. Then
    \[
        \frac{\|\lap^- f\|_{2}^2}{\dir^-(f)}
        = \frac{\|\lap_\pi f\|_{2}^2}{-\inp{f}{\lap^- f}}
        = \frac{\|\lap _\pi f\|_{2}^2}{\inp{f}{-\lap_\pi f}}
        \ge \lambda_\pi \,.
    \]
    Let $\lambda^* \define \min_\pi \lambda_\pi$, where $\pi$ ranges over all bijections from $V$ to $[n]$
    for which $-\lap_\pi$ has a non-zero eigenvalue. Then $\lambda^*$ is strictly positive, and for each
    $g \in L^2(V)$ with $\dir^-(g) > 0$, we have
    \[
        \frac{\|\lap^- g\|_{2}^2}{\dir^-(g)} \ge \lambda^* \,.
    \]
    Thus $\lambda^-(G) > 0$ as claimed.
\end{proof}

Next, we observe that $\lambda^-(G)$ characterizes the rate of exponential decay of $\dir^-(P_t f)$.
The key fact is that, in a gradient system, the rate of energy decay is governed by the rate of
change of the state.

\begin{proposition}
    \label{prop:rate-of-change-energy}
    Let $f \in L^2(V)$. Then the function $t \mapsto \dir^-(P_t f)$ is differentiable and, for all
    $t \ge 0$, 
    \[
        \odv*{\dir^-(P_t f)}{t} = -2\left\| \odv*{P_t f}{t} \right\|_{2}^2 \,.
    \]
\end{proposition}
\begin{proof}
    By the chain rule (for the Fréchet derivative) and \cref{prop:gradient-system},
    \[
        \odv*{\dir^-(P_t f)}{t}
        = \inp{\grad \dir^-(P_t f)}{\odv*{P_t f}{t}}
        = \inp{-2\odv*{P_t f}{t}}{\odv*{P_t f}{t}}
        = -2\left\| \odv*{P_t f}{t} \right\|_{2}^2 \,. \qedhere
    \]
\end{proof}

\begin{corollary}
    \label{prop:exponential-decay}
    Let $f \in L^2(V)$. Then the function $t \mapsto \dir^-(P_t f)$ satisfies, for all $t \ge 0$, 
    \begin{equation}
        \label{eq:directed-spectral-gap-decay}
        \odv*{\dir^-(P_t f)}{t} \le -2\lambda^-(G)\cdot \dir^-(P_t f) \,.
    \end{equation}
    Furthermore, $\lambda^-(G)$ is the largest constant for which
    \eqref{eq:directed-spectral-gap-decay} holds for all $f, t$.
\end{corollary}
\begin{proof}
    By \cref{prop:rate-of-change-energy},
    \[
        \odv*{\dir^-(P_t f)}{t}
        = -2\left\| \odv*{P_t f}{t} \right\|_{2}^2
        = -2\|\lap^- P_t f\|_{2}^2 \,.
    \]
    By \cref{def:directed-spectral-gap}, $\lambda^-(G) \le \|\lap^- P_t
    f\|_{2}^2/\dir^-(P_t f)$, whence \eqref{eq:directed-spectral-gap-decay} follows.
    Conversely, for every $\epsilon > 0$, by definition of infimum we can find $g \in L^2(V)$ such
    that $\|\lap^- g\|_{2}^2/\dir^-(g) < \lambda^-(G) + \epsilon$, in which case we
    have (recalling that $P_0 g = g$)
    \[
        \odv*{\dir^-(P_0 g)}{t}
        = -2\left\|\lap^- P_0 g\right\|_{2}^2
        > -2\left(\lambda^-(G)+\epsilon\right) \dir^-(P_0 g) \,,
    \]
    so the constant $\lambda^-(G)$ is indeed tight.
\end{proof}

We can now show that the dynamical spectral gap mediates an upper bound on the distance from the
initial state $f$ to its monotone equilibrium $P_\infty f$ in terms of the directed Dirichlet energy
of $f$, \ie a directed Poincaré inequality.

\begin{theorem}[Directed Poincaré inequality]
    \label{thm:directed-poincare-spectral}
    For all $f \in L^2(V)$, it holds that
    \[
        \|f - P_\infty f\|_{2}^2 \le \frac{1}{\lambda^-(G)} \, \dir^-(f) \,.
    \]
\end{theorem}
\begin{proof}
    Since $P_t f \to P_\infty t$, it suffices to show that $\|f - P_t f\|_{L^2(V)}^2$ satisfies the
    claimed upper bound for all $t \ge 0$. By the triangle inequality and
    \cref{prop:rate-of-change-energy},
    \[
        \|f - P_t f\|_{2}
        = \left\| \int_0^t \odv*{P_s f}{s} \odif s \right\|_{2}
        \le \int_0^t \left\| \odv*{P_s f}{s} \right\|_{2} \odif s
        = \int_0^t \sqrt{-\frac{1}{2} \cdot \odv*{\dir^-(P_s f)}{s}} \odif s \,.
    \]
    We now appeal to the exponential rate of decay of $\dir^-(P_t f)$ (\cref{prop:exponential-decay})
    via a straightforward calculus lemma stated below (\cref{lemma:exponential-decay-integral-sqrt})
    to conclude that
    \[
        \|f - P_t f\|_{2}
        \le \sqrt{\frac{1}{2}} \cdot \frac{2}{\sqrt{2\lambda^-(G)}} \sqrt{\dir^-(P_0 f)}
        = \frac{1}{\sqrt{\lambda^-(G)}} \sqrt{\dir^-(f)} \,. \qedhere
    \]
\end{proof}

We used the following lemma, whose proof we defer to \cref{appendix:lemmas}.

\begin{restatable}{lemma}{lemmaexponentialdecayintegralsqrt}
    \label{lemma:exponential-decay-integral-sqrt}
    Let $t > 0$ and suppose that $F : [0, t] \to [0, +\infty)$ is differentiable and satisfies
    $F'(s) \le -K F(s)$ for all $s \in [0,t]$, for some $K > 0$. Then it holds that
    \[
        \int_0^t \sqrt{-F'(s)} \odif s \le \frac{2}{\sqrt{K}} \sqrt{F(0)} \,.
    \]
\end{restatable}

\section{The dynamical spectral gap of the directed hypercube}\label{sec:directed-gap-of-hypercube}

Let $H_n$ denote the unweighted directed hypercube in dimension $n$, \ie $H_n = \left(\zo^n,
w\right)$ where the weight function $w$ is as follows: for each $x, y \in \zo^n$, $w(x, y) = 1$ if
$\|x-y\|_1 = 1$ with $x \preceq y$, and $w(x, y) = 0$ otherwise. For simplicity of notation, in this
section we also let $V \define \zo^n$.

This section studies the spectral gap of $H_n$ endowed with directed Laplacian operator $\lap^- =
\lap^-_{H_n}$ and associated directed Dirichlet energy functional $\dir^-$. We show

\begin{restatable}[Dynamical spectral gap of the directed hypercube]{theorem}{thmdirectedpoincarehn}
    \label{thm:directed-poincare-hn}
    $H_n$ satisfies $\lambda^-(H_n) = 1$.
\end{restatable}

Before proving the theorem, we first observe that $\lap^-$ and $\dir^-$ enjoy a useful
coordinate-wise decomposition: we write
\[
    \lap^- = \sum_{i=1}^n \lap^{(i)} \,,
\]
where each $\lap^{(i)} : L^2(V) \to L^2(V)$ is given by
\[
    (\lap^{(i)} f)(x)
    \define \frac{1}{2} \left( f(x^{\oplus i}) - f(x) \right) \ind{f(x^{i \to 0}) > f(x^{i \to 1})}
\]
for each $f \in L^2(V)$ and $x \in \zo^n$. It is straightforward to check that this decomposition
agrees with \cref{def:laplacian-general}. Similarly, from \cref{prop:energy-violations} we also
obtain the decomposition
\[
    \dir^- = \sum_{i=1}^n \dir^{(i)} \,,
\]
where each $\dir^{(i)} : L^2(V) \to \bR$ is given by
\[
    \dir^{(i)}(f) \define \frac{1}{4}\cdot \Exu{x \in \zo^n}{
        \left((f(x^{i \to 1}) - f(x^{i \to 0}))^-\right)^2}
\]
for each $f \in L^2(V)$. (The extra factor of $1/2$ compared to \cref{prop:energy-violations}
appears because the summation above counts each edge of $H_n$ twice.)

We now prove \cref{thm:directed-poincare-hn}. The upper bound is easy and attained by anti-dictator
functions, so the main point is to show the lower bound $\lambda^-(H_n) \ge 1$. To this end, we fix
any function $f$, consider the fraction in \cref{def:directed-spectral-gap}, and expand the
operators $\lap^-$ and $\dir^-$ according to the coordinate-wise decompositions above. The main
conceptual ingredient is the positive correlation $\inp{\lap^{(i)} f}{\lap^{(j)} f} \ge 0$,
which is established by an argument reminiscent of the analysis of the ``edge tester'' for
monotonicity of functions on the hypercube \cite{GGLRS00}\footnote{\cite{GGLRS00} define a
    \emph{switch operator} $S_i$ which fixes all violations of monotonicity of a Boolean function
    $f$ along direction $i$, by switching the values of $f$ along violating edges. A key lemma in
    that paper shows that the application of $S_i$ can only make the number of violations of
    monotonicity along a different direction $j$ smaller -- informally, the work along direction $i$
    ``helps'' toward the work along direction $j$. In our $L^2$ setting, this is captured by the
positive correlation $\inp{\lap^{(i)} f}{\lap^{(j)} f} \ge 0$ between the contributions of
directions $i$ and $j$ to the action of $\lap$.}\!\!.

\begin{proof}[Proof of \cref{thm:directed-poincare-hn}]
    \textbf{Upper bound.}
    Note that for the anti-dictator function $g(x) = -x_1$, the coordinate-wise decompositions give
    that $(\lap^- g)(x) = \pm \frac{1}{2}$ for every $x \in \zo^n$, while the $g$-antimonotone edges
    (which contribute to $\dir^-$) are precisely those along the first coordinate. Hence
    \[
        \frac{\|\lap^- g\|_{2}^2}{\dir^-(g)}
        = \frac{\Ex{(\pm 1/2)^2}}{\frac{1}{4} \Ex{(-1)^2}}
        = 1 \,,
    \]
    so $\lambda^-(H_n) \le 1$.

    \textbf{Lower bound.}
    Let $f \in L^2(V)$. It suffices to show that $\|\lap^- f\|_{2}^2 \ge \dir^-(f)$. Note that
    \[
        \left\|\lap^- f\right\|_{2}^2
        = \inp{\sum_{i=1}^n \lap^{(i)} f}{\sum_{i=1}^n \lap^{(i)} f}
        = \sum_{i=1}^n \left\|\lap^{(i)}f\right\|_{2}^2
            + 2\sum_{i<j} \inp{\lap^{(i)}f}{\lap^{(j)}f} ,
    \]
    and recall that $\dir^-(f) = \sum_{i=1}^n \dir^{(i)}(f)$. For each $i \in [n]$, we have
    \begin{align*}
        \left\|\lap^{(i)} f\right\|_{2}^2
        &= \Exu{x \in \zo^n}{\left((\lap^{(i)} f)(x)\right)^2}
        = \Exu{x \in \zo^n}{\left(\frac{1}{2}
                \left( f(x^{\oplus i}) - f(x) \right) \ind{f(x^{i \to 0}) > f(x^{i \to 1})}
        \right)^2} \\
        &= \frac{1}{4} \Exu{x \in \zo^n}{\left((f(x^{i \to 1}) - f(x^{i \to 0})^-)\right)^2}
        = \dir^{(i)}(f) \,.
    \end{align*}
    Therefore, we will be done if we can show that, for each $i \ne j$, it holds that
    $\inp{\lap^{(i)} f}{\lap^{(j)} f}\ge 0$:
    \begin{claim}
        \label{claim:i-j}
        For every $i, j \in [n]$ with $i \ne j$, we have
        \[
            \inp{\lap^{(i)} f}{\lap^{(j)} f} \ge 0 \,.
        \]
    \end{claim}
    We prove the claim below, but first we use it to conclude the proof of the theorem. Putting the
    above together, we obtain
    \[
        \left\|\lap^- f\right\|_{2}^2
        = \sum_{i=1}^n \left\|\lap^{(i)}f\right\|_{2}^2 + 2\sum_{i<j} \inp{\lap^{(i)}f}{\lap^{(j)}f}
        \ge \sum_{i=1}^n \dir^{(i)}(f)
        = \dir^-(f) \,.
    \]
    Since $f$ was arbitrary, we conclude that $\lambda^-(H_n) \ge 1$.
\end{proof}

\begin{proof}[Proof of \cref{claim:i-j}]
    Suppose without loss of generality that $i=1$ and $j=2$. We first observe that it suffices to
    consider each ``square'' obtained by fixing all but the first two coordinates, since the inner
    product decomposes along these squares. Concretely, for each $y \in \zo^{n-2}$, let $g_y:
    \zo^2\rightarrow\bR$ be given by $g_y(x) \define f(x, y)$ for each $x \in \zo^2$, where we write $f(x,
    y)$ for the value of $f$ at the input obtained by concatenating $x$ and $y$. Then
    \begin{align*}
        \inp{\lap^{(1)} f}{\lap^{(2)} f}
        &= \frac{1}{2^n} \sum_{z \in \zo^n}
            \left((\lap^{(1)} f)(z)\right) \left((\lap^{(2)} f)(z)\right) \\
        &= \frac{1}{2^n} \sum_{x \in \zo^2} \sum_{y \in \zo^{n-2}}
            \Bigg[\frac{1}{2}\left( f(x^{\oplus 1}, y) - f(x,y) \right)
                \ind{f(x^{1 \to 0}, y) > f(x^{1 \to 1}, y)}
            \\
        & \qquad\qquad\qquad\qquad\qquad\cdot \frac{1}{2} \left( f(x^{\oplus 2}, y) - f(x,y) \right)
                \ind{f(x^{2 \to 0}, y) > f(x^{2 \to 1}, y)}\Bigg]\\
        &= \frac{1}{2^n}
            \sum_{y \in \zo^{n-2}} \sum_{x\in\{0,1\}^{2}}\lap^{(1)} g_y(x)\cdot \lap^{(2)} g_y(x)
            \,.
    \end{align*}
    We will show that for any $g:\{0,1\}^{2}\rightarrow\bR$, the sum $\sum_{x\in\zo^{2}}\lap^{(1)} g(x)\cdot\lap^{(2)} g(x)
    $ is nonnegative, which will complete the proof. Define
    \begin{align*}
        c &\define g(0, 1) \,, & d &\define g(1, 1) \,, \\
        a &\define g(0, 0) \,, & b &\define g(1, 0) \,.
    \end{align*}
    Then we have
    \begin{align*}
        &\quad \sum_{x\in\{0,1\}^{2}}\lap^{(1)} g(x)\cdot \lap^{(2)} g(x)\\
        &= \sum_{x \in \zo^2}
            \left( g(x^{\oplus 1}) - g(x) \right) \ind{g(x^{1 \to 0}) > g(x^{1 \to 1})}
            \left( g(x^{\oplus 2}) - g(x) \right) \ind{g(x^{2 \to 0}) > g(x^{2 \to 1})} \\
        &= (b-a)\ind{a>b}(c-a)\ind{a>c}
                + (a-b)\ind{a>b}(d-b)\ind{b>d} \\
            &\qquad \qquad
                + (d-c)\ind{c>d}(a-c)\ind{a>c}
                + (c-d)\ind{c>d}(b-d)\ind{b>d} \\
        &= (a-b)^+(a-c)^+
                - (a-b)^+(b-d)^+
                - (c-d)^+(a-c)^+
                + (c-d)^+(b-d)^+ \\
        &= \left[(a-b)^+ - (c-d)^+\right] \left[ (a-c)^+ - (b-d)^+ \right] \ge 0 \,,
    \end{align*}
    where the last inequality is proved in \cref{lemma:abcd}.
\end{proof}

Combining \cref{thm:directed-poincare-hn,thm:directed-poincare-spectral}, we conclude

\begin{corollary}[Directed Poincaré inequality for the hypercube; refinement of \cref{thm:directed}]
    \label{cor:directed-poincare}
    For all $f \in L^2(V)$, it holds that
    \[
        \dist^\mono_2(f)^2
        \le \|f - P_\infty f\|_2^2
        \le \frac{1}{\lambda^-(H_n)} \, \dir^-(f)
        = \dir^-(f) \,.
    \]
\end{corollary}

\section{Necessity of approximate FKG}\label{sec:reverse}

The goal of this section is to prove \Cref{thm:spectral-gap-to-FKG}, a partial converse of \Cref{thm:FKG-to-spectral-gap}, which roughly states that for a monotone set $A\subseteq\zo^{n}$ to exhibit good spectral expansion, the approximate FKG ratio $\delta(A)$ must not be too close to $-1$. Analogous to the proof of \Cref{thm:FKG-to-spectral-gap}, the first step is to establish a reverse version of the directed Poincaré inequality, which we do in \Cref{subsec:reverse-directed-Poincare}. We then apply the argument from \Cref{sec:directed-to-undirected} in the reverse direction, as carried out in \Cref{subsec:sec-3-reverse}.

\subsection{Reverse directed Poincaré inequality}\label{subsec:reverse-directed-Poincare}
We give a reverse directed Poincaré inequality which establishes the tightness of
\cref{thm:directed} up to a factor of $n$. In fact, for a general directed graph $G$, the reverse
directed Poincaré inequality we show is mediated by the maximum total weighted degree of $G$.

Let $d_w(G)$ denote the maximum total weighted degree of any vertex in $G$:
\[
    d_w(G) \define \max_{u \in V} \left\{ \sum_{v \in V} (w(u,v) + w(v,u)) \right\} \,.
\]
We then show the following.

\begin{theorem}[Reverse directed Poincaré inequality]
    For all $f \in L^2(V)$ and all monotone $g \in L^2(V)$, we have
    \[
        \|f-g\|_{2}^2 \ge \frac{1}{d_w(G)} \, \dir^-(f) \,.
    \]
\end{theorem}
\begin{proof}
    Let $\Delta \in L^2(V)$ be given by $\Delta(u) \define \abs*{f(u) - g(u)}$ for each $u \in V$,
    so that $\|f-g\|_{2}^2 = \|\Delta\|_{2}^2$. For each $f$-antimonotone edge $(u,v)$ of
    $G$, we have $f(u) > f(v)$ while $g(u) \le g(v)$, which implies via the triangle inequality that
    \begin{equation}
        \label{eq:Delta}
        \Delta(u) + \Delta(v) \ge f(u) - f(v) \,.
    \end{equation}
    Using \cref{prop:energy-violations} and summing \eqref{eq:Delta} up over all $f$-antimonotone
    edges, we conclude that
    \begin{align*}
        \dir^-(f)
        &= \frac{1}{2} \Exu{u \in V}{
            \sum_{v \in V} w(u,v) \left(\left(f(u)-f(v)\right)^+\right)^2} \\
        &= \frac{1}{2} \Exu{u \in V}{
            \sum_{v \in V : f(u) > f(v)} w(u,v) \left(f(u)-f(v)\right)^2} \\
        &\le \frac{1}{2} \Exu{u \in V}{
            \sum_{v \in V : f(u) > f(v)} w(u,v) \left(\Delta(u)+\Delta(v)\right)^2} \\
        &\le \Exu{u \in V}{
            \sum_{v \in V : f(u) > f(v)} w(u,v) \left(\Delta(u)^2 + \Delta(v)^2\right)} \\
        &\le \Exu{u \in V}{\Delta(u)^2 \left(\sum_{v \in V}(w(u,v) + w(v,u))\right)} \\
        &\le d_w\cdot \|\Delta\|_{2}^2 \,. \qedhere
    \end{align*}
\end{proof}


The directed unweighted hypercube $H_n$ has $d_w(H_n) = n$. Hence the theorem above implies

\begin{corollary}\label{cor:reverse-directed-Poincare-cube}
    Fix graph $H_n$, the directed unweighted hypercube on vertex set $\zo^n$. Then for each $f \in
    L^2(\zo^n)$ we have
    \[
        \dist^\mono_2(f)^2 \ge \frac{1}{n} \, \cE^-(f) \,.
    \]
\end{corollary}

\subsection{Proof of \Cref{thm:spectral-gap-to-FKG}}\label{subsec:sec-3-reverse}

Having proved the reverse directed Poincar\'{e} inequality, we may now apply the argument from \Cref{sec:directed-to-undirected} to prove \Cref{thm:spectral-gap-to-FKG}.

First, recall that we defined an extension operator $T:L^{2}(A)\rightarrow L^{2}(\zo^{n})$ in \Cref{def:domain-extension}. We will need the following easy proposition which is complementary to \Cref{prop:dist-to-mono-under-extension}.

\begin{proposition}
Let $A\subseteq \zo^{n}$ be a monotone set with at least 2 elements. For every function $f:A\rightarrow\bR$ and any monotone increasing function $g:A\rightarrow \bR$, we have
$$\mu(A)^{-1/2}\cdot\dist_{2}^{\mono}(T[f])\leq \left\|f-g\right\|_{2}$$ where the $L^{2}$-norm is the norm in the inner product space $L^{2}(A)$.
\end{proposition}

\begin{proof}
Consider a function $\widetilde{g}:\zo^{n}\rightarrow\bR$ defined by
$$\widetilde{g}(x)=\begin{cases}
\max\{g(x),\min_{y\in A}f(y)\},&\text{if }x\in A,\\
\min_{y\in A}f(y),&\text{if }x\not\in A.
\end{cases}$$
By \Cref{def:domain-extension} we know that $T[f](x)=\widetilde{g}(x)$ for $x\in \zo^{n}\setminus A$. Furthermore, since $g:A\rightarrow\bR$ is increasing, it is easy to see that $\tilde{g}:\zo^{n}\rightarrow\bR$ is also increasing. Therefore we have
\begin{align*}
\dist_{2}^{\mono}(T[f])&\leq \sqrt{\Exu{x\in \zo^{n}}{\big(T[f](x)-\widetilde{g}(x)\big)^{2}}}=\sqrt{\mu(A)\cdot \Exu{x\in A}{\big(T[f](x)-\widetilde{g}(x)\big)^{2}}}\\
&\leq\sqrt{\mu(A)\cdot \Exu{x\in A}{\big(f(x)-g(x)\big)^{2}}}=\mu(A)^{1/2}\cdot\left\|f-g\right\|_{2}.\qedhere
\end{align*}
\end{proof}
By \Cref{def:tau}, we immediately have the following corollary.

\begin{corollary}\label{cor:dist-to-mono-under-extension-reverse}
Let $A\subseteq \zo^{n}$ be a monotone set with at least 2 elements. For every function $f:A\rightarrow\bR$ and any monotone increasing function $g:A\rightarrow \bR$, we have
$$\mu(A)^{-1/2}\cdot\dist_{2}^{\mono}(T[f])\leq \tau(f,g).$$
\end{corollary}
We are now ready to prove \Cref{thm:spectral-gap-to-FKG}.

\begin{proof}[Proof of \Cref{thm:spectral-gap-to-FKG}]
We divide into two cases.

\textbf{Case 1:} $
\delta(A)=0$. We pick an arbitrary monotone increasing function $g:A\rightarrow\bR$, and we claim that $n\cdot \Vars{A}{g}\geq \cE_{A}(g)$, certifying \Cref{thm:spectral-gap-to-FKG}. In fact, we have
\begin{align}
\quad\cE_{A}(g)&=\mu(A)^{-1}\cdot\cE^{-}(T[g])+\mu(A)^{-1}\cdot\cE^{-}(T[-g])&\text{(by \Cref{prop:Undirected-A-to-directed-full})}\nonumber\\
&\leq \frac{n}{\mu(A)}\cdot\dist_{2}^{\mono}(T[g])^{2}+\frac{n}{\mu(A)}\cdot\dist_{2}^{\mono}(T[-g])^{2}&\text{(by \Cref{cor:reverse-directed-Poincare-cube})}\nonumber\\
&\leq n\cdot\tau(g,g)^2+n\cdot \tau(-g,g)^{2} &\text{(by \Cref{cor:dist-to-mono-under-extension-reverse})}\nonumber\\
&=0+n\cdot \Vars{A}{g}=n\cdot\Vars{A}{g} &\text{(by \Cref{prop:tau-to-variance}).}\label{eq:reverse-everything}
\end{align}

\textbf{Case 2:} $\delta(A)<0$. Since the set of all monotone increasing functions $g:A\rightarrow\bR$ such that $\Vars{A}{g}=1$ is compact in the Euclidean space $\bR^{A}$, there exist functions $g,h$ in this set such that $\Covs{A}{g}{h}$ attains its minimum value, subject to $g$ and $h$ being monotone increasing and having variance 1. By \Cref{def:FKG-ratio}, it is easily seen that this minimum value equals $\delta(A)$.

We then consider the function $g-h$. Since $\Covs{A}{g}{h}=\delta(A)<0$, it is easy to see that $g-h$ is a non-constant function. We claim that the function $g-h$ certifies \Cref{thm:spectral-gap-to-FKG}, i.e.  
\begin{equation}\label{eq:reverse-goal}
(1+\delta(A))\cdot n\cdot \Vars{A}{g-h}\geq \cE_{A}(g-h).
\end{equation}
In fact, we have
\begin{align*}
&\quad\cE_{A}(g-h)=\mu(A)^{-1}\cdot\cE^{-}(T[g-h])+\mu(A)^{-1}\cdot\cE^{-}(T[h-g])&\text{(by \Cref{prop:Undirected-A-to-directed-full})}\\
&\leq \frac{n}{\mu(A)}\cdot\dist_{2}^{\mono}(T[g-h])^{2}+\frac{n}{\mu(A)}\cdot\dist_{2}^{\mono}(T[h-g])^{2}&\text{(by \Cref{cor:reverse-directed-Poincare-cube})}\\
&\leq n\cdot \tau(g-h,g)^{2}+n\cdot \tau(h-g,h)^{2} &\text{(by \Cref{cor:dist-to-mono-under-extension-reverse})}\\
&=n\cdot \left(2-\max\{0,\rho(g-h,g)\}^{2}-\max\{0,\rho(h-g,h)\}^{2}\right)\cdot \Vars{A}{g-h}&\text{(by \Cref{prop:tau-to-variance}).}
\end{align*}
We then calculate $\rho(g-h,g)$ by \Cref{def:rho}. Since $\Vars{A}{g}=\Vars{A}{h}=1$ and $\Covs{A}{g}{h}=\delta(A)$, we have
$$\rho(g-h,h)=\frac{\Covs{A}{g-h}{g}}{\sqrt{\Vars{A}{g-h}}}=\frac{1-\delta(A)}{\sqrt{2-2\delta(A)}}=\sqrt{\frac{1-\delta(A)}{2}}\geq 0.$$
Similarly, we have $\rho(h-g)=\sqrt{(1-\delta(A))/2}\geq 0$. Plugging the values of $\rho(g-h,g)$ and $\rho(h-g,h)$ into \eqref{eq:reverse-everything} yields the desired inequality \eqref{eq:reverse-goal}.
\end{proof}

\iftoggle{anonymous}{}{%
\section*{Acknowledgments}

We thank Lap Chi Lau for many helpful insights and discussions on spectral theory for directed
graphs, and Esty Kelman for early discussions in the development of this project. Renato Ferreira
Pinto Jr.\ is supported by an NSERC Canada Graduate Scholarship.
}

\printbibliography

\appendix

\section{Technical lemmas}
\label{appendix:lemmas}

\lemmaexponentialdecayintegralsqrt*
\begin{proof}
    We first observe that we may assume that $F(s) > 0$ for all $s \in (0, t)$. Indeed, if this is
    not the case, then let $s^* \define \inf\{s \in (0,t) : F(s) = 0\}$. By continuity of $F$, we
    obtain that $F(s^*) = 0$; then, since $F \ge 0$ and $F' \le 0$, we conclude that $F(s) = F'(s) =
    0$ for all $s \in [s^*, t]$. Thus it suffices to prove the claim for the interval $[0, s^*]$,
    and indeed $F(s) > 0$ for all $s \in (0, s^*)$. Thus assume without loss of generality that
    $F(s) > 0$ for all $s \in (0, t)$.

    By the chain rule, the function $s \mapsto \sqrt{F(s)}$ is differentiable with
    $\odv*{\sqrt{F(s)}}{s} = \frac{F'(s)}{2\sqrt{F(s)}} \le 0$ for all $s \in (0,t)$, and thus
    \begin{align*}
        \sqrt{-F'(s)}
        &= -\frac{2\sqrt{F(s)}}{\sqrt{-F'(s)}} \frac{F'(s)}{2\sqrt{F(s)}}
        = \frac{2\sqrt{F(s)}}{\sqrt{-F'(s)}} \left(-\odv*{\sqrt{F(s)}}{s}\right) \\
        &\le \frac{2\sqrt{F(s)}}{\sqrt{K F(s)}} \left(-\odv*{\sqrt{F(s)}}{s}\right)
        = -\frac{2}{\sqrt{K}} \odv*{\sqrt{F(s)}}{s} \,.
    \end{align*}
    Therefore
    \[
        \int_0^t \sqrt{-F'(s)} \odif s
        \le -\frac{2}{\sqrt{K}} \int_0^t \left(\odv*{\sqrt{F(s)}}{s}\right) \odif s
        = \frac{2}{\sqrt{K}} \left(\sqrt{F(0)} - \sqrt{F(t)}\right)
        \le \frac{2}{\sqrt{K}} \sqrt{F(0)} \,. \qedhere
    \]
\end{proof}

\begin{lemma}
    \label{lemma:abcd}
    For any $a, b, c, d \in \bR$, we have
    \[
        \left[(a-b)^+ - (c-d)^+\right] \left[ (a-c)^+ - (b-d)^+ \right]\geq 0\,.
    \]
\end{lemma}
\begin{proof}
    Let $X \define (a-b)^+ - (c-d)^+$ and $Y \define (a-c)^+ - (b-d)^+$, so that our goal is to show
    that $XY \ge 0$, or equivalently that $X > 0 \implies Y \ge 0$ and $Y > 0 \implies X \ge 0$. By
    switching the roles of $b$ and $c$, it suffices to prove the first implication, \ie that $X > 0
    \implies Y \ge 0$.

    Suppose $X > 0$. Then $a - b > 0$, since otherwise we would have $(a-b)^+ = 0$ and hence $X \le
    0$. Therefore $a-b = (a-b)^+ > (c-d)^+ \ge c-d$. Moreover, if $(b-d)^+ = 0$ then $Y \ge 0$ and
    we are done, so we may assume that $b-d = (b-d)^+ > 0$. We conclude that
    \[
        Y
        = (a-c)^+ - (b-d)^+
        \ge (a-c) - (b-d)
        = (a-b) - (c-d)
        \ge 0 \,,
    \]
    where the first inequality holds since $(a-c)^+ \ge a-c$ while $(b-d)^+ = b-d$, and the second
    inequality holds since $a-b \ge c-d$.
\end{proof}

\end{document}